\documentclass[11pt, twoside, leqno]{article}
\usepackage{amssymb,latexsym}
\usepackage{enumerate}
\usepackage{color}
\usepackage{amsmath, amsthm, amsfonts, amssymb, mathrsfs}
\newtheorem{theorem}{Theorem}[section]
\newtheorem{corollary}[theorem]{Corollary}
\newtheorem{lemma}[theorem]{Lemma}
\newtheorem{proposition}[theorem]{Proposition}

\newtheorem{remark}[theorem]{Remark}
\newtheorem*{thma}{\hskip\parindent {Theorem A}}
\newtheorem*{thmb}{\hskip\parindent {Theorem B}}

\numberwithin{equation}{section}

\numberwithin{equation}{section}

\begin{document}

\baselineskip=17pt

\title{Jump and variational inequalities for rough operators}

\date{}

\maketitle

 \begin{center}
{\bf Yong Ding}\\
School of Mathematical Sciences\\
Beijing Normal University\\
Laboratory of Mathematics and Complex Systems (BNU)\\
Ministry of Education of China\\
Beijing 100875, China\\
E-mail: {\it dingy@bnu.edu.cn} \vskip 0.5cm

{\bf Guixiang Hong}\\
School of Mathematics and Statistics\\
Wuhan University\\
Wuhan 430072, China\\
and\\
Instituto de Ciencias Matem\'aticas\\
CSIC-UAM-UC3M-UCM\\
 Consejo Superior de Investigaciones
Cient\'ificas\\
Madrid 28049, Spain\\
E-mail: {\it guixiang.hong@icmat.es} \vskip 0.5cm

{\bf Honghai Liu}\\
School of Mathematics and Information Science,\\
Henan Polytechnic University,\\
Jiaozuo, Henan, 454003,  China \\
E-mail: {\it hhliu@hpu.edu.cn}
\end{center}

\renewcommand{\thefootnote}{}

\footnote{2010 \emph{Mathematics Subject Classification}: Primary
42B25; Secondary 42B20.}

\footnote{\emph{Key words and phrases}: Jump inequalities,
Variational inequalities, Singular integrals, Averaging operators,
Rough kernels}


\renewcommand{\thefootnote}{\arabic{footnote}}
\setcounter{footnote}{0}

\newpage

\begin{abstract}
In this paper, we systematically study jump and variational inequalities for rough operators, whose research have been initiated by Jones {\it et al}. More precisely, we show some jump and variational inequalities for the families $\mathcal T:=\{T_\varepsilon\}_{\varepsilon>0}$ of truncated singular integrals and $\mathcal M:=\{M_t\}_{t>0}$ of averaging operators with rough kernels, which are defined  respectively by
$$
T_\varepsilon
f(x)=\int_{|y|>\varepsilon}\frac{\Omega(y')}{|y|^n}f(x-y)dy
$$
 and
$$M_t f(x)=\frac1{t^n}\int_{|y|<t}\Omega(y')f(x-y)dy,
$$
where the kernel $\Omega$ belongs to $L\log^+\!\!L(\mathbf S^{n-1})$ or  $H^1(\mathbf S^{n-1})$ or $\mathcal{G}_\alpha(\mathbf S^{n-1})$ (the condition introduced by Grafakos and Stefanov). Some of our results are sharp in the sense that the underlying assumptions are the best known conditions for the boundedness of corresponding maximal operators.

 \end{abstract}

 \section{Introduction}\label{sect1}
 The jump and variational inequalities have been the subject of many
 recent articles in probability, ergodic theory and harmonic
 analysis. The first variational inequality was proved by L\'epingle \cite{Lep76} for martingales (see \cite{PiXu88} for a simple proof). Bourgain \cite{Bou89} is the first one using L\'epingle's result to obtain similar variational estimates for the ergodic averages, and then directly deduce pointwise convergence results without previous knowledge that pointwise convergence holds for a dense subclass of functions, which are not available in some ergodic models. In particular, Bourgain's work \cite{Bou89} has inaugurated a new research direction in ergodic theory and harmonic
analysis. In their papers \cite{JKRW98} \cite{JRW03} \cite{JRW00} \cite{CJRW2000}, \cite{CJRW2002}, Jones and his collaborators systematically studied jump and variational inequalities for ergodic averages and truncated singular integrals (mainly of homogeneous type). Since then many other publications came to enrich the literature on this subject (cf. e.g. \cite{GiTo04}, \cite{LeXu2}, \cite{DMT12}, \cite{JSW08}, \cite{MaTo},\cite{OSTTW12}, \cite{Hon1}). Recently, several works on weighted and vector-valued jump and variational inequalities in ergodic theory and harmonic analysis have also appeared (cf. e.g. \cite{MTX}, \cite{KZ}, \cite{HLP}, \cite{Hon}, \cite{HoMa}).

Let us first recall some definitions and known results, then state our results.

1. \emph{$q$-variation norm $\|\mathfrak{a}\|_{V_q}$ of  a family $\mathfrak{a}$ of complex numbers }

Given a family of complex numbers $\mathfrak{a}=\{a_t: t\in\mathbb{R}\}$ and $q\ge1$, the $q$-variation norm of the family $\mathfrak{a}$ is defined by
\begin{equation}\label{q-ver number family}
\|\mathfrak{a}\|_{V_q}=\sup\big(|a_{t_0}|+\sum_{k\geq1}
|a_{t_k}-a_{t_{k-1}}|^q\big)^{\frac{1}{q}},
\end{equation}
where the supremum runs over all increasing sequences $\{t_k:k\geq0\}$. It is trivial that
\begin{equation}\label{number contr ineq}
\|\mathfrak{a}\|_{L^\infty(\mathbb{R})}:=\sup_{t\in\mathbb R}|a_t|
\le\|\mathfrak{a}\|_{V_q}\quad\text{for}\ \ q\ge1.
\end{equation}

2. \emph{Strong $q$-variation function $V_q(\mathcal F)(x)$ of { a family $\mathcal F$ of functions} }

Via the definition \eqref{q-ver number family}  of the $q$-variation norm of a family of numbers, one may define the strong $q$-variation function $V_q(\mathcal F)$ of a family $\mathcal F$ of functions.
Given a family of Lebesgue measurable functions $\mathcal F=\{F_t:t\in\mathbb{R}\}$ defined on $\mathbb{R}^n$, for fixed $x$ in $\mathbb{R}^n$, the value of the strong $q$-variation function $V_q(\mathcal F)$ of the family $\mathcal F$ at $x$ is defined by
\begin{equation}\label{defini of Vq(F)}
V_q(\mathcal F)(x)=\|\{F_t(x)\}_{t\in\mathbb{R}}\|_{V_q},\quad q\ge1.
\end{equation}

 Usually, the measurability of the strong $q$-variation function is not automatically available. However,  all the strong $q$-variation function considered in the present paper are measurable, see e.g. \cite{CJRW2000} or \cite{JSW08} for some explanations.

 Suppose $\mathcal{A}=\{{A}_t\}_{t>0}$ is a family of operators  on $L^p(\Bbb R^n)\, (1\le p\le\infty)$. The strong $q$-variation operator is simply defined as
$$V_q(\mathcal Af)(x)=\|\{A_t(f)(x)\}_{t>0}\|_{V_q},\quad\forall f\in L^p(\mathbb{R}^n).$$
It is easy to observe from the definition of $q$-variation norm that for any $x$ if $V_q(\mathcal Af)(x)<\infty$, then $\{A_t(f)(x)\}_{t>0}$ converges when $t\rightarrow0$ or $t\rightarrow\infty$. In particular, if $V_q(\mathcal Af)$ belongs to some function spaces such as $L^p(\mathbb{R}^n)$ or $L^{p,\infty}(\mathbb{R}^n)$, then the sequence converges almost everywhere without any additional condition. This is why mapping property of strong $q$-variation operator is so interesting in ergodic theory and harmonic analysis.

3. \emph{$\lambda$-jump function $N_\lambda(\mathcal F)$ of a family $\mathcal F$ of functions}

Given a family of Lebesgue measurable functions $\mathcal F=\{F_t:t\in\mathbb{R}\}$ defined on $\mathbb{R}^n$, for $\lambda>0$ and $x\in\mathbb{R}^n$, the value of the $\lambda$-jump function $N_\lambda(\mathcal F)$ at $x$ is defined by
$$N_\lambda(\mathcal F)(x)=\sup\big\{N\in\Bbb N:\ \exists\ s_1<t_1\leq s_2<t_2\leq\dotsc\leq s_N<t_N\ \text{such that\ } |F_{t_k}(x)-F_{s_k}(x)|>\lambda\big\}.$$

 By \cite{BP}, for a function family $\mathcal F=\{F_t:t\in\mathbb{R}\}$, $\lambda>0$ and $q\ge1$, the $\lambda$-jump function $N_\lambda(\mathcal F)$ is pointwisely controlled by the strong $q$-variation $V_q(\mathcal F)$ in the following sense,
 \begin{equation}\label{contr ineq}
\lambda(N_\lambda(\mathcal F)(x))^{1/q}\leq C_qV_q(\mathcal F)(x),\quad x\in\mathbb{R}^n,
\end{equation}
where $C_q=2^{1+1/q}$.
On the other hand, in 2008, Jones, Seeger and  Wright gave the following result.

\begin{lemma} {\rm (\cite{JSW08})} \label{lemma1}
Let $p_0<\rho<p_1$ and $\mathcal{A}=\{{A}_t\}_{t>0}$ be a family of operators. If for all $p_0<p<p_1$,
$$\sup_{\lambda>0}\|\lambda\big(N_\lambda(\mathcal{A}f)\big)^{1/\rho}\|_p\leq C_p\|f\|_p,$$
then for all $p_0<p<p_1$ and for all $q>\rho$, $\|V_{q}(\mathcal{A}f)\|_p\leq C_{p,q}\|f\|_p$ .
\end{lemma}

This lemma gives a converse inequality of \eqref{contr ineq} in some sense.

4. \emph{Family $\mathcal T$ of truncated singular integral operators}

Recall the Calder\'on-Zygmund singular integral operator $T$ with homogeneous kernel is defined by
\begin{equation}\label{SIO}
T(f)(x)={\rm p.v.}\int_{\mathbb{R}^n}\frac{\Omega(y')}{|y|^n}f(x-y)dy,
\end{equation}
where $\Omega\in L^1({\mathbf S}^{n-1})$
 satisfies the cancelation condition
 \begin{equation}\label{can of O}
\int_{\mathbf S^{n-1}}\Omega(y')d\sigma(y')=0.
 \end{equation}
 Let $\mathcal T=\{T_\varepsilon\}_{\varepsilon>0}$, where $T_\varepsilon$ is the truncated operator of $T$ defined by
 \begin{equation}\label{tr of S}
T_\varepsilon
f(x)=\int_{|y|>\varepsilon}\frac{\Omega(y')}{|y|^n}f(x-y)dy.
 \end{equation}

The famous Hilbert transform $H$, which is defined by
\begin{equation}\label{Hilbert trans}
H(f)(x)={\rm p.v.}\frac 1\pi\int_{\mathbb{R}}\frac{f(y)}{x-y}dy,
\end{equation}
is the example of the homogeneous singular integral operator $T_\Omega$ when the dimension $n=1$.

In 2000, Campbell {\it et al} \cite{CJRW2000} first considered the $L^p(\mathbb R)\,(1<p<\infty)$ boundedness of the strong $q$-variation operator of the family of the truncated Hilbert transforms denoted by $\mathcal H:=\{H_\varepsilon\}_{\varepsilon>0}$.
As a corollary, the authors of \cite{CJRW2000} obtained the boundedness of $\lambda$-jump function for Hilbert transform associated to $q$ with $q>2$. The proof depends on a special property of Hilbert transform $H$. That is, $H_\varepsilon$ can be written as a combinations of certain convolution operators, which in turn can be written as combinations of differential operators.

In 2002, Campbell {\it et al} \cite{CJRW2002} gave the $L^p(\mathbb R^n)$ boundedness of the strong $q$-variation operator of  $\mathcal T$, the family of homogenous singular integrals with $\Omega\in L\log^+\!\!L(\mathbf S^{n-1})$ and $n\geq2$.

\begin{thma} {\rm (\cite{CJRW2002})}
Suppose $\Omega$ satisfies \eqref{can of O} and $\Omega\in L\log^+\!\!L(\mathbf S^{n-1})$. If $q>2$, then for $1<p<\infty$ and $f\in L^p(\mathbb R^n)$, the strong $q$-variation function $V_q(\mathcal Tf)(x)\in L^p(\mathbb R^n)$. In particular,
$$ \|V_q(\mathcal Tf)\|_{L^p(\mathbb R^n)}\le
C_{p,q,n}\|f\|_{L^p(\mathbb R^n)},
$$
where and in the sequel, the constant $C_{p,q,n}>0$ depends only on $q,p,n$.
\end{thma}
 The basic idea of proving Theorem A is classical, that is, the authors applied the Calder\'on-Zygmund rotation method to reduce the desired variational inequalities to the one dimensional results. The $\lambda$-jump inequalities associated to $q>2$ were obtained as a corollary by (\ref{contr ineq}).

In 2008, using the Fourier transform and square function estimates given in \cite{DR86},
 Jones, Seeger and Wright \cite{JSW08} developed a general method, which allows one to obtain some jump inequalities for the family of truncated singular integral operators $\mathcal T=\{T_\varepsilon\}_{\varepsilon>0}$ and the other integral operators arising from harmonic analysis.

\begin{thmb} {\rm (\cite{JSW08})}
Suppose $\Omega$ satisfies \eqref{can of O} and $\Omega\in L^r(\mathbf S^{n-1})$ for $r>1$. Then $\lambda$-jump inequality  $\sup_{\lambda>0}\|\lambda \sqrt{N_\lambda(\mathcal T f)}\|_{L^p(\mathbb R^n)}\le
C_{p,n}\|f\|_{L^p(\mathbb R^n)}\ (1<p<\infty)$ holds,
where and in the sequel, the constant $C_{p,n}>0$ depends only on $p$ and $n$.
\end{thmb}

Notice the following well known inclusion relations between some function spaces on $\mathbf S^{n-1}$:
\begin{equation}\label{including rel}
L^\infty(\mathbf S^{n-1})\subsetneq L^r(\mathbf S^{n-1})\,(1<r<\infty)\subsetneq L\log^+\!\!L(\mathbf S^{n-1})\subsetneq H^1(\mathbf S^{n-1})\subsetneq L^1(\mathbf S^{n-1}),
\end{equation}
where and in the sequel, $H^1(\mathbf  S^{n-1})$ denotes the Hardy space on $\mathbf S^{n-1}$, the definition and some facts on $H^1(\mathbf  S^{n-1})$ can be found in \cite{C82}, \cite{LDY} and \cite{RW}.

Thus, \eqref{including rel} inspires us to consider the  question:  whether the conclusions of Theorem A and Theorem B  are still true if $\Omega\in H^1(\mathbf S^{n-1})$? The first main result in this paper gives an answers of the above question.

\begin{theorem}\label{thm:LlogL}
Let $\mathcal T$ be the family of truncated singular integral operators given in \eqref{tr of S} and $\Omega$ satisfies \eqref{can of O}.

{\rm (i)} If $\Omega\in H^1(\mathbf S^{n-1})$, then for all $q>2$ and $1<p<\infty$, the following strong $q$-variation inequality holds
\begin{equation}\label{Vqr}
 \|V_q(\mathcal Tf)\|_{L^p(\mathbb R^n)}\le
C_{p,q,n}\|f\|_{L^p(\mathbb R^n)}.
\end{equation}

{\rm (ii)} If $\Omega\in L\log^+\!\! L(\mathbf S^{n-1})$, then the following $\lambda$-jump inequality holds
\begin{equation}\label{N12}
\sup_{\lambda>0}\|\lambda \sqrt{N_\lambda(\mathcal T f)}\|_{L^p(\mathbb R^n)}\le
C_{p,n}\|f\|_{L^p(\mathbb R^n)}\quad  \text {for all}\ 1<p<\infty.
\end{equation}

\end{theorem}

\begin{remark}
{\rm From the relation \eqref{including rel}, it is easy to see that the conclusion (i) of Theorem \ref{thm:LlogL} is an essential improvement of Theorem A by \eqref{including rel}. Similarly, the conclusion (ii)  is also an essential improvement of Theorem B. }
\end{remark}

\begin{remark}
{\rm It is well-known that $\Omega\in H^1(\mathbf S^{n-1})$ is the best condition for the boundedness of maximal singular integral operator. However, we would like to point out that it is still not clear whether the estimate (\ref{N12}) holds for $\Omega\in H^1(\mathbf S^{n-1})$ up to now. So, this remains an open problem.}
\end{remark}

In 1998, Grafakos and Stefanov \cite{GS98} introduced alternative conditions on $\Omega$ to study $L^p$ boundedness of maximal singular integral operators. More precisely, they considered the family of conditions
\begin{align}\label{GS condition}
\mathcal{G}_\alpha(\mathbf S^{n-1})=\bigg\{\Omega:\ \sup_{\xi\in\mathbf S^{n-1}}\int_{\mathbf S^{n-1}}|\Omega(\theta)|\big(\log\frac{1}{|\theta\cdot\xi|}\big)^{1+\alpha}d\theta<\infty\bigg\},\quad \alpha>0.
\end{align}
 It is known from \cite{GS98} and \cite{GS99} that
 \begin{equation*}
\begin{split}
\bigcup_{r>1}L^r(\mathbf S^{n-1})&\subsetneq \bigcap_{\alpha>0} \mathcal{G}_\alpha(\mathbf
 S^{n-1}),\\
\bigcap_{\alpha>0}\mathcal{G}_\alpha(\mathbf
 S^{n-1})&\nsubseteq L\log^+\!\!L(\mathbf
 S^{n-1})\nsubseteq \bigcap_{\alpha>0} \mathcal{G}_\alpha(\mathbf
 S^{n-1}),\\
\bigcap_{\alpha>0} \mathcal{G}_\alpha(\mathbf
 S^{n-1})&\nsubseteq H^1(\mathbf
 S^{n-1}).
\end{split}
\end{equation*}
Hence, it is of interest to ask whether
the jump and variational inequalities hold for the family  $\mathcal T$ of truncated singular integral operators if $\Omega\in \mathcal{G}_\alpha$ for some $\alpha>0$? This constitutes the second result of this paper.

\begin{theorem}\label{thm:Omega alpha}
Let $\mathcal T$ be the family of truncated singular integral operators given in \eqref{tr of S}, $\Omega$ satisfies \eqref{can of O} and $\Omega\in \mathcal{G}_\alpha(\mathbf S^{n-1})$ for some $\alpha>1$. Then
the following $\lambda$-jump inequality
\begin{equation}\label{N15}
\sup_{\lambda>0}\|\lambda \sqrt{N_\lambda(\mathcal T f)}\|_{L^p(\mathbb R^n)}\le
C_{p,n}\|f\|_{L^p(\mathbb R^n)}
\end{equation}
holds for all $(3+\alpha)/(1+\alpha)<p<(3+\alpha)/2$. In particular, if $\Omega\in\bigcap_{\alpha>1}\mathcal{G}_\alpha(\mathbf S^{n-1})$, then \eqref{N15} holds for all $1<p<\infty$.
\end{theorem}

The associated strong $q$-variation  inequality is an immediate consequence of Theorem \ref{thm:Omega alpha} by Lemma \ref{lemma1}.

\begin{corollary}\label{thm:Omega alpha vq}
Under the same conditions as Theorem \ref{thm:Omega alpha}, the strong $q$-variation  inequality
$$\|V_q(\mathcal Tf)\|_{L^p(\mathbb R^n)}\le
C_{p,q,n}\|f\|_{L^p(\mathbb R^n)}
$$
holds for all $q>2$ and  $(3+\alpha)/(1+\alpha)<p<(3+\alpha)/2$. In particular, the above estimate holds for  $\Omega\in\bigcap_{\alpha>1}\mathcal{G}_\alpha(\mathbf S^{n-1})$ and all $1<p<\infty$.
\end{corollary}

\begin{remark}
{\rm In \cite{FGP}, the author showed the associated maximal singular integral operator, that is the strong $\infty$-variation operator,  is $L^p$-bounded for $(1+2\alpha)/(2\alpha)<p<1+2\alpha$ and $\alpha>1/2$. So
we think the scope of $p$ in Corollary \ref{thm:Omega alpha vq} is not optimal, and it is very interesting to enlarge the scope of $p$ depending on $q$.}
\end{remark}

The main sketch of proving Theorem \ref{thm:LlogL} and Theorem \ref{thm:Omega alpha} are taken from  \cite{CJRW2002} and \cite{JSW08}. That is, we first reduce the $\lambda$-jump estimate to short $2$-variation estimate and dyadic $\lambda$-jump estimate (see the beginning of the next section for related definitions and statement); then use the rotation method or the vector-valued singular integral operator theory to deal with short $2$-variation operators; and use Fourier transform and square function estimates to obtain dyadic $\lambda$-jump estimate.

 However, the underlying details are substantially different due to the kernels being quite rough. For instance, firstly when $\Omega\in L\log^+L(\mathbf{S}^{n-1})$, it follows from \cite{DR86} that unlike the case $\Omega\in L^r(\mathbf{S}^{n-1})$ ($r>1$) the decay of Fourier transform of associated kernel is not available any more, and we have to decompose $\Omega$ into pieces having polynomial decay as done in \cite{AP02}.  In order to finally sum all the pieces to conclude the desired result, we exploit some subtle calculations to get sharp bounds for each piece in obtaining dyadic $\lambda$-jump estimate.

Secondly when $\Omega\in \mathcal{G}_\alpha(\mathbf{S}^{n-1})$, to obtain the short $2$-variation estimate, first of all, the rotation method seems not to work here. Instead, we appeal to the vector-valued singular integral operator theory, which has been made use of by Jones {\it et al} \cite{JSW08} for averaging over spheres and curves. But in our case, the kernel being rough brings us a lot of difficulties in both obtaining the $L^2$ estimate and verifying the H\"ormander condition.

5. \emph{Family $\mathcal M$ of the averaging operators with rough kernel}

The third aim of the present paper is to establish some jump and variational inequalities for the family
$\mathcal M=\{M_t\}_{t>0}$. Here $M_t$ denotes the averaging operator with rough kernels defined by
\begin{equation}\label{averaging operator}
M_t f(x)=\frac1{t^n}\int_{|y|<t}\Omega(y')f(x-y)dy,
\end{equation}
where $\Omega\in L^1({\mathbf S}^{n-1})$.

The motivation of considering the family $\mathcal M=\{M_t\}_{t>0}$ is two-fold. Firstly, in the case $\Omega\equiv1$,
the jump and variational inequalities for the family $\mathcal M=\{M_t\}_{t>0}$ have been well studied in \cite{CJRW2000} and \cite{JRW00}, which were based on the 1-dimensional results \cite{Bou89} and \cite{JKRW98}. Secondly,
the maximal operator associated with  $\mathcal M=\{M_t\}_{t>0}$
\begin{equation}\nonumber
M^*(f)(x)=\sup_{t>0}\frac1{t^n}\int_{|y|<t}|\Omega(y')f(x-y)|dy,
\end{equation}
plays a very important role in studying rough singular integral operators (see
\cite{ST}, \cite{DR86}, \cite{CR}, \cite{D93} or \cite{LDY} for more details).

\begin{theorem}\label{thm:maximal}
Suppose the family $\mathcal M=\{M_t\}_{t>0}$ is defined in \eqref{averaging operator}.

{\rm (i)}  If $\Omega\in H^1(\mathbf S^{n-1})$ or $L(\log^+\!\!L)^{1/2}(\mathbf S^{n-1})$ and
$1<p<\infty$, then the $\lambda$-jump inequality for the family $\mathcal M$ holds. That is, there
exists $C_{p,n}>0$ so that
\begin{equation}\label{N12m}
\sup_{\lambda>0}\|\lambda \sqrt{N_\lambda(\mathcal M f)}\|_{L^p(\mathbb R^n)}\le
C_{p,n}\|f\|_{L^p(\mathbb R^n)}.
\end{equation}

{\rm(ii)} If $\Omega\in L^1(\mathbf S^{n-1})$, then for $q>2$ and $1<p<\infty$,  the strong $q$-variation  inequality for the family $\mathcal M$ holds. That is, there
exists $C_{p,q,n}>0$ so that
\begin{equation}\label{Vqrm}
 \|V_q(\mathcal Mf)\|_{L^p(\mathbb R^n)}\le
C_{p,q,n}\|f\|_{L^p(\mathbb R^n)}.
\end{equation}
\end{theorem}

The idea of the proof of Theorem \ref{thm:maximal} is similar as that for Theorem \ref{thm:LlogL}. That is, we use the rotation method to show (\ref{Vqrm}) and discrete Marcinkiewicz integrals to deal with associated dyadic $\lambda$-jump estimate.

\begin{remark}
\emph{Note that both $H^1(\mathbf S^{n-1})$ and $L(\log^+\!\!L)^{1/2}(\mathbf S^{n-1})$ contain $L\log^+\!\!L(\mathbf S^{n-1})$, but they do not contain each other (see e.g. \cite{AACP}). It seems difficult to get \eqref{N12m} using only the method in the present paper for $\Omega\in L^1(\mathbf S^{n-1})$, since which is not sufficient for boundedness of Marcinkiewicz integrals.}
\end{remark}

The proofs of Theorem \ref{thm:LlogL}, Theorem \ref{thm:Omega alpha} and Theorem \ref{thm:maximal} will be given in Sections 2, 3 and 4, respectively.

\section{Proof of Theorem \ref{thm:LlogL}}
We shall mainly show Theorem \ref{thm:LlogL}(ii), the $\lambda$-jump estimate \eqref{N12}, while Theorem \ref{thm:LlogL}(i) will be obtained in the course of the proof.
Let us start with recalling the strategy taken by Jones {\it et al} \cite{JSW08}.
Let $\mathcal F=\{F_t\}_{t\in\mathbb R^+}$ be a family of functions. For $j\in\mathbb Z$ and $x\in\mathbb R^n$, define
 $$
V_{2,j}(\mathcal F)(x)=\bigg(\sup_{\substack
{t_1<\cdots<t_N\\
[t_l,t_{l+1}]\subset[2^j,2^{j+1}]}}\sum_{l=1}^{N-1}|F_{t_{l+1}}(x)-F_{t_l}(x)|^2\bigg)^{1/2}
 $$
 and the \emph{short $2$-variation operator}
 $$
S_2(\mathcal F)(x)=\bigg(\sum_{j\in\mathbb Z}[V_{2,j}(\mathcal
F)(x)]^2\bigg)^{1/2}.
 $$
We also define \emph{dyadic
$\lambda$-jump function} as follows:
$$N_{\lambda}^{d}(\mathcal F)(x)=\sup_{N\in\Bbb N}\big\{\exists\ j_1<k_1\le j_2<k_2\le \cdots\le j_N<k_N, \text{s.t.}\ |F_{2^{k_l}}(x)-F_{2^{j_l}}(x)|>\lambda\big\}.$$

Then the following pointwise comparison holds.
 \begin{lemma}\ {\rm (see \cite[Lemma 1.3]{JSW08})}\label{lem:convert lemma}
$$\lambda\sqrt{N_\lambda(\mathcal F)(x)}\le
C\bigg(S_2(\mathcal F)(x)+\lambda\sqrt{N_{\lambda/3}^{d}(\mathcal
F)(x)}\bigg)$$
uniformly in $\lambda>0$.
 \end{lemma}
Lemma \ref{lem:convert lemma} reduces the desired estimate
\begin{align*}
\|\lambda \sqrt{N_\lambda(\mathcal Af)}\|_{L^p(\mathbb R^n)}\le
C_{p}\|f\|_{L^p(\mathbb R^n)}
\end{align*}
 for any fixed $1<p<\infty$
to
\begin{align}\label{dyadic jump singular}
\|\lambda \sqrt{N^d_\lambda(\mathcal Af)}\|_{L^p(\mathbb R^n)}\le
C_{p}\|f\|_{L^p(\mathbb R^n)}
\end{align}
and
\begin{align}\label{short variation singular}
\|S_2(\mathcal Af)\|_{L^p(\mathbb R^n)}\le
C_{p}\|f\|_{L^p(\mathbb R^n)}
\end{align}
for any family of linear operators $\mathcal{A}$.

Let $\mathcal{T}$ be a family of truncated singular integrals of homogeneous type. In \cite{CJRW2002}, the authors established estimate (\ref{short variation singular}) for $\Omega\in L\log^+L(\mathbf{S}^{n-1})$ using the rotation method. In the following subsection, we observe that the rotation method works also in the case $\Omega\in H^1(\mathbf{S}^{n-1})$. Whence we obtain the strong $q$-strong estimate (\ref{Vqr}), which is Theorem \ref{thm:LlogL}(i).

While in \cite{JSW08}, the authors showed estimate (\ref{dyadic jump singular}) for $\Omega\in L^r(\mathbf{S}^{n-1})$ making use of the polynomial decay of Fourier transform of the measure $\nu$ defined by
$$\langle \nu, f\rangle=\int_{1\leq|x|\leq2}\frac{\Omega(x/|x|)}{|x|^n}f(x)dx.$$ In the second subsection, we show that actually estimate (\ref{dyadic jump singular}) is still true for $\Omega\in L\log^+L(\mathbf{S}^{n-1})$ even though  the polynomial decay of Fourier transform of the measure $\nu$ is not available as explained in the Introduction. Thus we obtain \eqref{N12}, which is Theorem \ref{thm:LlogL}(ii).

\subsection{Proof of Theorem \ref{thm:LlogL}(i)}
As observed in \cite{CJRW2002}, the rotation method allows us to obtain simultaneously short $2$-variation estimate and strong $q$-variation estimate since both estimates for the family of truncated Hilbert transforms have been established in \cite{CJRW2000}. In \cite{CJRW2002}, they provided the proof of strong $q$-variation estimate. In the present paper, we give a sketch of the proof of short $2$-variation.

\begin{proposition}\label{pro:short H1}
Let $\mathcal T$ be given as
in \eqref{tr of S}. $\Omega$ satisfies \eqref{can of O} and $\Omega\in H^1(\mathbf S^{n-1})$. Then
$$
\|S_2(\mathcal T f)\|_{L^p(\mathbb R^n)}\le C_p\|f\|_{L^p(\mathbb
R^n)}\quad\text{for}\ \ 1<p<\infty.
$$
The same inequality holds for the strong $q$-variation operator $V_q(\mathcal Tf)$ with $q>2$.
\end{proposition}
 To prove Proposition \ref{pro:short H1} by the Calder\'on-Zygmund rotation method, we need the following known result in one dimension case.

\begin{lemma}\label{S21}{\rm (\cite[Theorem 1.4]{CJRW2000})}
Let $\mathcal H=\{H_\varepsilon\}_{\varepsilon>0}$ be the family of truncated Hilbert transforms. Then
$$\|S_2(\mathcal H
f)\|_{L^p(\mathbb R)}\le C_p\|f\|_{L^p(\mathbb R)}\quad\text{for}\ \ 1<p<\infty.$$
Similar statement holds true for $V_q(\mathcal Hf)$ with $q>2$.
\end{lemma}

\begin{proof}
Let us turn to the proof of Proposition \ref{pro:short H1}.  Since $\Omega\in H^1(\mathbf S^{n-1})$ and satisfies \eqref{can of O}, it is easy to check that we may write
$\Omega=\Omega_o+\Omega_e$, where $\Omega_o$ is odd and
$\Omega_e$ is even, and both of them belonging to $H^1(\mathbf S^{n-1})$ as well as satisfying cancelation condition (\ref{can of O}). Thus, we need only to prove Lemma \ref{pro:short H1} for $\Omega$ is odd and even, respectively. The following proof is based on the
idea of \cite{CJRW2002}, we cite some estimates there for the sake
of completeness.

 {\bf Case 1.} $\Omega$ is odd. For an
interval $I\subset(0,\infty)$, let
$$
H_I^1f(x)=\int_{|s|\in I}\frac{f(x-s\mathbf{1})}{s}ds,
$$
where $\mathbf{1}=(1,0,\cdots,0)$. For simplicity, we denote
$H^1_{(\varepsilon,\infty)}$ by $H^1_\varepsilon$, and set $\mathcal
H^1=\{H^1_\varepsilon\}_{\varepsilon>0}$. Let $y'=y/|y|\in\mathbf
S^{n-1}$ and $\sigma$ be the rotation on $\mathbb{R}^n$. We define the rotation of a function by
$(R_{\sigma} f)(x)=f(\sigma x)$. Let $d\sigma$ denote Haar measure on $SO(n)$, normalized
so that $\int_{SO(n)}d\sigma=|\mathbf S^{n-1}|$, the Lebesgue measure of $\mathbf S^{n-1}$. Then by \cite[p.222]{SW}, we have
\begin{equation}
\int_{|y|\in I}\frac{\Omega(y')}{|y|^n}f(x-y)dy=\frac12\int_{SO(n)}(R_{\sigma^{-1}}H_I^1R_{\sigma} f)(x)\Omega(\sigma\mathbf 1)d\sigma(y').
\end{equation}
For fixed $j\in\mathbb Z$, there is a partition $\{I_{i,j}\}=\{I_{i,j}(x)\}$ of $[2^j,2^{j+1}]$ such that (see \cite{CJRW2002})
\begin{align*}
V_{2,j}(\mathcal T f)(x)\le 2\bigg[\sum_{i}\big|\frac12\int_{
SO(n)}(R_{\sigma^{-1}}H_{I_{i,j}}^1R_{\sigma}
f)(x)\Omega(\sigma\mathbf 1)d\sigma(y')\big|^2\bigg]^{1/2}.
\end{align*}
Further, by Minkowski's inequality, we get the following estimate
\begin{equation}\label{rot}
S_2(\mathcal T f)(x)\le \int_{SO(n)}R_{\sigma^{-1}}S_2(\mathcal H^1R_{\sigma} f)(x)|\Omega(\sigma\mathbf
1)|d\sigma(y').
\end{equation}
For
$1<p<\infty$, using Minkowski's inequality and Lemma \ref{S21}, we have
\begin{align}
\nonumber\|S_2(\mathcal T f)\|_{L^p(\mathbb R^n)}&\le \int_{SO(n)}\bigg(\int_{\mathbb R^n}R_{\sigma^{-1}}\big|S_2(\mathcal
H^1R_{\sigma}
f)(x)\big|^pdx\bigg)^{1/p}|\Omega(\sigma\mathbf 1)|d\sigma(y')\\
\label{S2Lp}&\le C_p\int_{SO(n)}\|R_{\sigma}
f\|_{L^p(\mathbb R^n)}|\Omega(\sigma\mathbf 1)|d\sigma(y')\\
&\le
C_p\|\Omega\|_{L^1(\mathbf S^{n-1})}\|f\|_{L^p(\mathbb R^n)}.\nonumber
\end{align}

{\bf Case 2.} $\Omega$ is even. It is well known that
$f=-\sum_{i=1}^nR^2_if$ for any Schwartz function $f$, where $R_i$ denotes the Riesz transform. For fixed
$f$, we denote $-R_if$ by $g_i$. Let $\Phi$ be a infinitely differentiable function on $\mathbb{R}$
such that $\Phi(t)=0$ if $t<\frac14$, and $\Phi(t)=1$ if
$t>\frac34$. Moreover, $0\le \Phi(t)\le1$ for all $t\in\mathbb R$. Define
$$
M_i(x)=\lim_{\varepsilon\rightarrow0}\int_{|x-y|>\varepsilon}\frac{\Omega(y)}{|y|^n}\Phi(|y|)
\frac{x_i-y_i}{|x-y|^{n+1}}dy.
$$
By \cite[p.302]{CZ}, we know that
$$\int_{\mathbb
R^n}\frac{\Omega(x-y)}{|x-y|^n}\Phi\big(\frac{|x-y|}\varepsilon\big)f(y)dy=\frac1{\varepsilon^n}\int_{\mathbb
R^n}\sum_{i=1}^nM_i\big(\frac{x-y}\varepsilon\big)g_i(y)dy.$$
Hence we have
\begin{align*}
T_\varepsilon f(x)&=\sum_{i=1}^n\frac1{\varepsilon^n}\int_{\mathbb
R^n}M_i\big(\frac{x-y}\varepsilon\big)g_i(y)dy
-\int_{|x-y|<\varepsilon}\frac{\Omega(x-y)}{|x-y|^n}\Phi(\frac{|x-y|}\varepsilon)f(y)dy\\
&=:\sum_{i=1}^nM_{i,\varepsilon}\ast g_i(x)-T^0_\varepsilon f(x).
\end{align*}
Note that $S_2$ is subadditive, then
$$
S_2(\mathcal T f)(x)\le \sum_{i=1}^nS_2(\{M_{i,\varepsilon}\ast
g_i\})(x)+S_2(\{T^0_\varepsilon f\})(x).
$$
To estimate $S_2(\{M_{i,\varepsilon}\ast g_i\})\ (i=1,2,\cdots,n)$,
it suffices to consider the case $i=1$ and the other cases can be treated similarly.
Define
\begin{equation}\label{comp of R and T}
N(x)=\lim_{\varepsilon\rightarrow0}\int_{|x-y|>\varepsilon}\frac{\Omega(y)}{|y|^n}\frac{x_1-y_1}{|x-y|^{n+1}}dy.
\end{equation}
Write
\begin{align*}
M_{1,\varepsilon}\ast
g_1(x)&=\frac1{\varepsilon^n}\int_{|y|>\varepsilon}N(\frac y{\varepsilon})g_1(x-y)dy+\frac1{\varepsilon^n}
\int_{|y|\le\varepsilon}N\big(\frac{y}\varepsilon\big)\Phi\big(\frac{y}\varepsilon\big)g_1(x-y)dy\\
&\qquad+\frac1{\varepsilon^n}\int_{|y|\le\varepsilon}\big[M_1\big(\frac{y}\varepsilon\big)
-N\big(\frac{y}\varepsilon\big)\Phi\big(\frac{y}\varepsilon\big)\big]g_1(x-y)dy\\
&\qquad+\frac1{\varepsilon^n}\int_{|y|>\varepsilon}\big[M_1\big(\frac{y}\varepsilon\big)-N\big(\frac{y}\varepsilon\big)\big]g_1(x-y)dy\\
&=:N_\varepsilon\ast g_1(x)+N^{\Phi}_{\varepsilon}\ast
g_1(x)+\Delta_\varepsilon\ast g_1(x)+D_\varepsilon\ast g_1(x).
\end{align*}
 Hence, we are reduced to estimate the $L^p$ norm of the following terms
 $$
 S_2(\{N_\varepsilon\ast g_1\})(x),\ S_2(\{N^{\Phi}_{\varepsilon}\ast
g_1\})(x),\ S_2(\{\Delta_\varepsilon\ast g_1\})(x),\
S_2(\{D_\varepsilon\ast g_1\})(x),\ S_2(\{T^0_\varepsilon f\})(x).
 $$
The estimate of $\ S_2(\{N_\varepsilon\ast g_1\})$ depends on the following result.

\begin{lemma}\label{prop of N}{\rm(\cite{GS99})}
Suppose $\Omega$ satisfies \eqref{can of O} and belongs to
$H^1(\mathbf S^{n-1})$, $N(x)$ is given in \eqref{comp of R and T}.
Then $N$ satisfies the following properties:
\begin{itemize} \item[{\rm(i)}] $N(x)$ is a homogeneous function of
order $-n$ on $\mathbb{R}^n$;
\item[{\rm(ii)}] $N(-x)=-N(x),\ \  x\in \mathbb{R}^n$;
\item[{\rm(iii)}] $\displaystyle\int_{\mathbf S^{n-1}}\left|N(x')\right|d\sigma(x')\leq
C\|\Omega\|_{H^1(\mathbf S^{n-1})}$.
\end{itemize}
\end{lemma}

By the result showed in Case 1, Lemma \ref{prop of N} and the $L^p\,(1<p<\infty)$ boundedness of Riesz transform, we obtain the
estimate
\begin{align*}
\|S_2(\{N_\varepsilon\ast g_1\})\|_{L^p(\mathbb R^n)}\le
C_p\|g_1\|_{L^p(\mathbb R^n)}\le C_p\|f\|_{L^p(\mathbb R^n)}.
\end{align*}
\par
To estimate the other four terms, we give a lemma.
\begin{lemma}\label{S2varphi}
For any $u\in\mathbf S^{n-1}$, let $\varphi_u$ be a function defined on
$[0,\infty)$ and $\mathcal
F_u(s)=\{\frac1{\varepsilon}\varphi_{u}(\frac s{\varepsilon})\}_{\varepsilon>0}$.
If there exists a constant $C_p$ independent of $u$ such that
$$
\|S_2(\mathcal F_u\ast h)\|_{L^p(\mathbb R)}\le
C_p\|h\|_{L^p(\mathbb R)},\ 1<p<\infty,
$$
then
$$
\|S_2(\mathcal Gf)\|_{L^p(\mathbb R^n)}\le C_p\|f\|_{L^p(\mathbb
R^n)}, 1<p<\infty,
$$
where $\mathcal Gf=\{G_\varepsilon f\}_{\varepsilon>0}$ and
$$G_\varepsilon f(x)=\int_{\mathbf
S^{n-1}}\Omega(u)\int_0^\infty\varphi_u(t)f(x-\varepsilon
tu)dtd\sigma(u)$$ with $\Omega\in L^1(\mathbf S^{n-1})$.
\end{lemma}

\begin{proof} In the same way as the proof of (\ref{rot}), we have
\begin{align*}
S_2(\mathcal Gf)(x)\le\int_{SO(n)}|\Omega(\sigma\mathbf 1)|R_{\sigma^{-1}}S_2(\mathcal I^{\mathbf 1}R_{\sigma} f)(x)d\sigma(u),
\end{align*}
where $\sigma$ is the rotation such that $\sigma\mathbf
1=u$,
$$
\mathcal I^{\mathbf 1}f(x)=\{ I^{\mathbf 1}_\varepsilon
f(x)\}_{\varepsilon>0}\ \ and \ \ I^{\mathbf 1}_\varepsilon
f(x)=\int_0^\infty \frac1\varepsilon\varphi_{\mathbf
1}(r/\varepsilon)f(x-r{\mathbf 1})dr.
$$
By Minkowski's inequality and hypothesis, we have
\begin{align*}
\|S_2(\mathcal Gf)\|_{L^p(\mathbb R^n)}&\le\int_{SO(n)}|\Omega(\sigma\mathbf 1)|\bigg(\int_{\mathbb R^n}
|S_2(\mathcal I^{\mathbf 1}R_{\sigma}
f)(x)|^pdx\bigg)^{1/p}d\sigma(u)\\
&\le C_p\|\Omega\|_{L^1(\mathbf
S^{n-1})}\|f\|_{L^p(\mathbb R^n)}.
\end{align*}
\end{proof}
\par
Now we continue the proof of Proposition \ref{pro:short H1}. It was proved that all the kernels of $T^0_\varepsilon$,
$N^{\Phi}_{\varepsilon}$, $\Delta_\varepsilon$ and $D_\varepsilon$
have a representation that satisfies the hypothesis of Lemma 3.4  in
\cite[p.2124]{CJRW2002}, and
$$
\int_0^1t|\varphi_u'(t)|dt<\infty.
$$
 Thus, we get
$$\|S_2(\mathcal F_u\ast
h)\|_{L^p(\mathbb R)}\le C_p\|h\|_{L^p(\mathbb R)}$$ for $1<p<\infty$
according to Lemma 2.4 in \cite{CJRW2000}. Therefore, Lemma \ref{S2varphi} and the $L^p$ boundedness of Riesz transforms imply that
\begin{align*}
&\|S_2(\{N^{\Phi}_{\varepsilon}\ast g_1\})\|_{L^p(\mathbb
R^n)}+\|S_2(\{\Delta_\varepsilon\ast g_1\})\|_{L^p(\mathbb
R^n)}+\|S_2(\{D_\varepsilon\ast g_1\})\|_{L^p(\mathbb
R^n)}+\|S_2(\{T^0_\varepsilon f\})\|_{L^p(\mathbb R^n)}\\
&\le
C_p\|f\|_{L^p(\mathbb R^n)}.
\end{align*}
This completes the proof of Proposition \ref{pro:short H1} and \eqref{Vqr} follows.
\end{proof}

\subsection{Proof of Theorem \ref{thm:LlogL}(ii)}

As explained at the beginning of this section, it suffices to deal with dyadic $\lambda$-jump estimate. 
\begin{proposition}\label{pro:dyadic LlogL}
Let $\mathcal T$ be given as
in \eqref{tr of S} with $\Omega\in L\log^+L(\mathbf S^{n-1})$ satisfying \eqref{can of O}. Then
\begin{align*}
\|\lambda \sqrt{N^d_\lambda(\mathcal Tf)}\|_{L^p(\mathbb R^n)}\le
C_{p}\|f\|_{L^p(\mathbb R^n)}.
\end{align*}
\end{proposition}
\begin{proof}
For
$j\in\mathbb Z$,  define a measure $\nu_{j}$ by
$$
\nu_j\ast f(x)=\int_{2^j\le |y|<2^{j+1}}\frac{\Omega(y)}{|y|^n}f(x-y)dy.
$$
Obviously, for $k\in\mathbb Z$,
$$
T_{2^k}f(x)=\int_{|y|\ge
2^k}\frac{\Omega(y)}{|y|^n}f(x-y)dy=\sum_{j\ge k}\nu_j\ast f(x).
$$
Let $\phi\in \mathscr{S}(\mathbb R^n)$ be a radial function such that $\hat{\phi}(\xi)=1$ for
$|\xi|\le2$ and $\hat{\phi}(\xi)=0$ for $|\xi|>4$. We have the following decomposition
\begin{align*}
T_{2^k}f&=\phi_k\ast Tf-\phi_k\ast\sum_{l<0}\nu_{k+l}\ast
f+\sum_{s\ge0}(\delta_0-\phi_k)\ast\nu_{k+s}\ast f\\
&:=T^1_kf-T^2_kf+T_k^3f,
\end{align*}
where $\phi_k$ satisfies $\widehat{\phi_k}(\xi)=\hat{\phi}(2^k\xi)$ and $\delta_0$ is the Dirac measure at 0.
$\mathscr T^if$ denotes the family $\{T^i_kf\}_{k\in\mathbb Z}$  for
$i=1,2,3$. Obviously, to show Proposition \ref{pro:dyadic LlogL} it suffices to prove the following inequalities:
\begin{align}\label{i=1,2,3}
\|\lambda [N_{\lambda/3}^{d}(\mathscr T^if)]^{1/2}\|_{L^p(\mathbb
R^n)}\le C_p\|f\|_{L^p(\mathbb R^n)},\ \ 1<p<\infty,\ \ i=1,2,3.
\end{align}

\medskip

\noindent {\bf Estimate of \eqref{i=1,2,3} for $i=1$.}
We need the following result, which has been essentially included in Theorem 1.1(i) of \cite{JSW08}. However some details there were omitted, we add these details for completeness.
\begin{lemma}\label{var of the 1.1}
Let $\phi_k$ be given above, set $\mathscr Uf=\{\phi_k\ast f\}_{k}$. Then for $1<p<\infty$,
\begin{align}\label{Uf estimate}
\|\lambda\sqrt{N_\lambda(\mathscr Uf)}\|_{L^p(\mathbb R^n)}\le C_p\|f\|_{L^p(\mathbb R^n)} \end{align}
holds uniformly in $\lambda>0$.
\end{lemma}
If we accept Lemma \ref{var of the 1.1} for a moment and combine it with the $L^p$-boundedness of $T$ (see \cite{CZ}), we can get the following estimate easily
\begin{align*}
\|\lambda[N^d_\lambda(\mathscr T^1f)]^{1/2}\|_{L^p(\mathbb
R^n)}=\|\lambda[N_\lambda(\{\phi_k\ast Tf\})]^{1/2}\|_{L^p(\mathbb
R^n)}\le C_p\| Tf\|_{L^p(\mathbb R^n)}\le C_p\|f\|_{L^p(\mathbb R^n)}.
\end{align*}
Hence, to estimate \eqref{i=1,2,3} for $i=1$ it remains to prove Lemma \ref{var of the 1.1}.
\begin{proof}
We borrow some notations and results from \cite[pp.6724]{JSW08}. For $j\in\mathbb Z$ and $\beta=(m_1,\cdots,m_n)\in\mathbb Z^n$,  we denote the dyadic cube $\prod_{k=1}^n(m_k2^j,(m_k+1)2^j]$ in $\mathbb R^n$ by $Q_\beta^j$, and the set of all dyadic cubes with side-length $2^j$ by $\mathcal D_j$. The conditional expectation of a local integrable $f$ with respect to $\mathcal D_j$ is given by
$$
\mathbb E_jf(x)=\sum_{Q\in \mathcal D_j}\frac1{|Q|}\int_{Q}f(y)dy\cdot\chi_{Q}(x)
$$
for all $j\in\mathbb Z$.
Note that $N_\lambda$ is subadditive, then
\begin{equation}\nonumber
N_\lambda(\mathscr Uf)\le N_{\lambda/2}(\mathscr Df)+N_{\lambda/2}(\mathscr Ef),
\end{equation}
where
$$
\mathscr Df=\{\phi_k\ast f-\mathbb E_kf\}_k \quad \text{and}\quad  \mathscr Ef=\{\mathbb E_kf\}_k.
$$
The following inequality is the $\lambda$-jump estimate for dyadic martingales (see, for instance, \cite{PiXu88}),
$$
\|\lambda\sqrt{N_{\lambda/2}(\mathscr{E}f)}\|_{L^p(\mathbb R^n)}\leq C_p\|f\|_{L^p(\mathbb R^n)},\ 1<p<\infty,
$$
 uniformly in $\lambda$. Next, observe that
\begin{equation}\nonumber
\lambda\sqrt{N_{\lambda/2}(\mathscr{D}f)}\le C\big(\sum_{k\in\mathbb Z}|\phi_k\ast f-\mathbb E_kf|^2\big)^{1/2}:=\mathcal Sf.
\end{equation}
\par
On the other hand, for $x\in{\Bbb R}^n$ and $j\in\Bbb Z$, let
$\mathbb D_jf(x)=\mathbb E_{j}f(x)-\mathbb E_{j-1}f(x)$. Thus, by $f\in L^p({\Bbb R}^n)$ and the Lebesgue differential theorem we see that
$$f(x)=-\sum_j\mathbb D_jf(x)\qquad \text{a. e.}\ x\in{\Bbb R}^n.$$
By Lemma 3.2 in \cite{JSW08}, we have
\begin{align}
\nonumber\|\mathcal Sf\|_{L^2(\mathbb R^n)}&\le \big(\sum_k(\sum_j\|\phi_k\ast \mathbb D_jf(x)-\mathbb E_k\mathbb D_jf\|_{L^2(\mathbb R^n)})^2\big)^{1/2}\\
\nonumber&\le\big(\sum_k(\sum_j2^{-\theta|k-j|}\|\mathbb D_jf\|_{L^2(\mathbb R^n)})^2\big)^{1/2}\\
\label{S22}&\le C_\theta(\sum_j\|\mathbb D_jf\|_{L^2(\mathbb R^n)}^2)^{1/2}\le C_\theta\|f\|_{L^2(\mathbb R^n)}
\end{align}
for some $\theta>0$.
\par
Jones {\it et al} have proved the following weak-type $(1,1)$ estimate for $\mathcal S$,
\begin{equation}\label{w11}
\alpha |\{x\in \mathbb R^n:\mathcal Sf(x)>\alpha\}|\le C\|f\|_{L^1(\mathbb R^n)},
\end{equation}
which is (33) in \cite{JSW08}. By interpolation, (\ref{S22}), (\ref{w11}), imply all the $L^p$ bounds for $1<p\le2$.
\par
For $2<p<\infty$, by H\"{o}lder's inequality, we have
\begin{align*}
\big\|\big(\sum_{k\in\mathbb Z}|\phi_k\ast f-\mathbb E_kf|^2\big)^{1/2}\big\|_{L^p(\mathbb R^n)}&=\sup_{\|\{h_k\}\|_{L^{p'}(l^2)}\le1}\big|\int_{\mathbb R^n}\sum_{k}[\phi_k\ast f(x)-\mathbb E_kf(x)]h_k(x)dx\big|\\
&=\sup_{\|\{h_k\}\|_{L^{p'}(l^2)}\le1}\big|\int_{\mathbb R^n}\sum_{k}[\phi_k\ast h_k(y)-\mathbb E_kh_k(y)]f(y)dy\big|\\
&\le\sup_{\|\{h_k\}\|_{L^{p'}(l^2)}\le1}\|\sum_{k}[\phi_k\ast h_k-\mathbb E_kh_k]\|_{L^{p'}(\mathbb R^n)}\|f\|_{L^p(\mathbb R^n)}.
\end{align*}
It suffices to show that
\begin{align}\label{sdp'}
\|\sum_{k}[\phi_k\ast h_k-\mathbb E_kh_k]\|_{L^{p'}(\mathbb R^n)}\le C\|\{h_k\}\|_{L^{p'}(l^2)},\ \ 1<p'\le 2.
\end{align}
Clearly, the following estimate is a consequence  $L^2$ boundedness of $\mathcal S$ by duality
\begin{align}\label{sd2}
\|\sum_{k}[\phi_k\ast h_k-\mathbb E_kh_k]\|_{L^2(\mathbb R^n)}\le C\|\{h_k\}\|_{L^2(l^2)}.
\end{align}
Therefore, we just need to prove that, if $\{h_k\}\in L^1(l^2)$, then
\begin{align}\label{sd1}
\big|\{x\in\mathbb R^n:|\sum_{k}[\phi_k\ast h_k(x)-\mathbb E_kh_k(x)]|>\alpha\}\big|\le \frac C{\alpha}\|\{h_k\}\|_{L^1(l^2)},
\end{align}
where $\alpha>0$ and $C$ is independent of $\alpha$ and $\{h_k\}$.
In fact, by interpolation between (\ref{sd2}) and (\ref{sd1}), we get (\ref{sdp'}).

 \par
  For $\alpha>0$, we perform Calder\'{o}n-Zygmund decomposition of $\|\{h_k\}\|_{l^2}$ at height $\alpha$, then there exists $\Lambda\subseteq\mathbb Z\times\mathbb Z^n$ such that the collection of dyadic cubes $\{Q_\beta^j\}_{(j,\beta)\in\Lambda}$ are disjoint and the following hold:
  \begin{itemize}
  \item[{\rm(i)}]
  $|\bigcup_{(j,\beta)\in\Lambda}Q_\beta^j|\le \alpha^{-1}\|\{h_k\}\|_{L^1(l^2)}$;
  \item[{\rm(ii)}]
  $\|\{h_k(x)\}\|_{l^2}\le \alpha$, if $x\not\in\bigcup_{(j,\beta)\in\Lambda}Q_\beta^j$;
   \item[{\rm(iii)}]
   $\frac1{|Q_\beta^j|}\int_{Q_\beta^j}\|\{h_k(x)\}\|_{l^2}dx\le2^n\alpha$ for each $(j,\beta)\in\Lambda$.
   \end{itemize}
 For $k\in\mathbb Z$, we set
\begin{equation}\nonumber
 g^{(k)}(x)=\left\{
 \begin{array}{ll}
 h_k(x),& \text{if}\ x\not\in \bigcup_{(j,\beta)\in\Lambda}Q_\beta^j,\\
 \frac1{|Q_\beta^j|}\int_{Q_\beta^j}h_k(y)dy,& \text{if}\ x\in Q_\beta^j,(j,\beta)\in\Lambda.
 \end{array}
 \right.
 \end{equation}
 and
 \begin{equation}\nonumber
 b^{(k)}(x)=\sum_{(j,\beta)\in\Lambda}[h_k(x)-\mathbb E_jh_k(x)]\chi_{Q_\beta^j}(x):=\sum_{(j,\beta)\in\Lambda}b^{(k)}_{j,\beta}(x).
 \end{equation}
 First we have $\|\{g^{(k)}\}\|^2_{L^2(l^2)}\le 2\alpha\|\{h_k\}\|_{L^1(l^2)}$. In fact, by (ii),(iii) and Minkowski's inequality,
 \begin{align*}
 \|\{g^{(k)}\}\|^2_{L^2(l^2)}&=\int_{(\cup_{(j,\beta)\in\Lambda}Q_\beta^j)^c}\|\{h_k(x)\}\|^2_{l^2}dx+\sum_{(j,\beta)\in\Lambda}\int_{Q_\beta^j}\sum_k\big|\frac1{|Q_\beta^j|}\int_{Q_\beta^j}h_k(y)dy\big|^2dx\\
 &\le \alpha\int_{(\cup_{(j,\beta)\in\Lambda}Q_\beta^j)^c}\|\{h_k(x)\}\|_{l^2}dx+2^n\alpha\sum_{(j,\beta)\in\Lambda}\int_{Q_\beta^j}\|h_k(x)\|_{l^2}dx\\
 &\le 2^n\alpha\|\{h_k\}\|_{L^1(l^2)}.
 \end{align*}
 On the other hand, it is easy to see that
 $$\int_{\Bbb R^n}b^{(k)}_{j,\beta}(x)dx=0\quad \text{for all}\quad k\in\mathbb Z,\ (j,\beta)\in\Lambda,$$
 and $\|\{b^{(k)}\}\|_{L^1(l^2)}\le 2\|\{h_k\}\|_{L^1(l^2)}$. In fact, by (iii) and Minkowski's inequality,
 \begin{align*}
 \|\{b^{(k)}\}\|_{L^1(l^2)}&=\sum_{(j,\beta)\in\Lambda}\int_{Q_\beta^j}\big(\sum_k|h_k(x)-\mathbb E_jh_k(x)|^2\big)^{1/2}dx\\
 &\le \sum_{(j,\beta)\in\Lambda}\int_{Q_\beta^j}(\sum_k|h_k(x)|^2)^{1/2}dx+\sum_{(j,\beta)\in\Lambda}\int_{Q_\beta^j}\big(\sum_k|\frac1{|Q_\beta^j|}\int_{Q_\beta^j}h_k(y)dy|^2\big)^{1/2}dx\\
 &\le 2\|\{h_k\}\|_{L^1(l^2)}.
 \end{align*}
Thus, for above $\alpha$, by (\ref{sd2}),

\begin{align*}
\alpha^2\big|\{x\in \mathbb R^n:|\sum_{k}[\phi_k\ast g^{(k)}(x)-\mathbb E_kg^{(k)}(x)]|>\alpha\}\big|&\le C\big\|\sum_{k}[\phi_k\ast g^{(k)}-\mathbb E_kg^{(k)}]\big\|_{L^2(\mathbb R^n)}^2\\
&\le C\|\{g^{(k)}\}\|_{L^2(l^2)}^2\le C\alpha\|\{h_k\}\|_{L^1(l^2)}.
\end{align*}

So, we get
\begin{equation}\nonumber
\big|\{x\in \mathbb R^n:|\sum_{k}[\phi_k\ast g^{(k)}(x)-\mathbb E_kg^{(k)}(x)]|>\alpha\}\big|\le \frac C{\alpha}\|\{h_k\}\|_{L^1(l^2)}.
\end{equation}
Let $\tilde{Q}_\beta^j$ be the cube concentric with $Q_\beta^j$ and with side length $4$ times that of $Q_\beta^j$. It is obvious that
\begin{equation}
\big|\bigcup_{(j,\beta)\in\Lambda}\tilde{Q}_\beta^j\big|\le C\sum_{(j,\beta)\in\Lambda}|Q_\beta^j|\le \frac C{\alpha}\|\{h_k\}\|_{L^1(l^2)}.
\end{equation}

Note that $\mathbb E_kb^{(k)}_{j,\beta}$ is supported in $Q_\beta^j$ when $k\le j$ and $\mathbb E_kb^{(k)}_{j,\beta}$ vanishes everywhere when $k\ge j$.

\begin{align*}
\alpha\big|\{x\not\in\bigcup\tilde{Q}_\beta^j:|\sum_{k}[\phi_k\ast b^{(k)}(x)-\mathbb E_kb^{(k)}(x)]|>\alpha\}\big|\le C\sum_{(j,\beta)\in\Lambda}\sum_k\int_{(\tilde{Q}_\beta^j)^c}|\phi_k\ast b_{j,\beta}^{(k)}(x)|dx.
\end{align*}
By \cite[(34) and (35),\,p.6726]{JSW08}, the inequality
\begin{align*}
\int_{(\tilde{Q}_\beta^j)^c}|\phi_k\ast b_{j,\beta}^{(k)}(x)|dx\le C2^{-\delta|j-k|}\int_{\mathbb R^n}|b_{j,\beta}^{(k)}(x)|dx
\end{align*}
holds for some $\delta>0$.
Applying H\"{o}lder's inequality, we get
\begin{align*}
\sum_{(j,\beta)\in\Lambda}\sum_k\int_{(\tilde{Q}_\beta^j)^c}|\phi_k\ast b_{j,\beta}^{(k)}(x)|dx&\le C\sum_{(j,\beta)\in\Lambda}\sum_k2^{-\delta|j-k|}\int_{\mathbb R^n}|b_{j,\beta}^{(k)}(x)|dx\\
&\le C\sum_{(j,\beta)\in\Lambda}\int_{\mathbb R^n}\big(\sum_k|b_{j,\beta}^{(k)}(x)|^2\big)^{1/2}dx\\
&\le C\sum_{(j,\beta)\in\Lambda}\|b_{j,\beta}\|_{L^1(l^2)}\le C\|\{h_k\}\|_{L^1(l^2)}.
\end{align*}
This completes the proof of Lemma \ref{var of the 1.1}.
\end{proof}

\medskip

\noindent {\bf Estimate of \eqref{i=1,2,3} for $i=3$.} First, let us give a decomposition of $\Omega\in L\log^+\!\!L(\mathbf S^{n-1})$ satisfying \eqref{can of O},
which can be found in \cite{AP02} or \cite{AACP}.
For $m\in\mathbb N$, denote
  $$E_m=\{y'\in \mathbf S^{n-1}:2^{m-1} \le|\Omega(y')|<2^{m}\}$$
and
$$
\Omega_m=\|\Omega\|_{L^1(E_m)}^{-1}\bigg[\Omega\chi_{E_m}-\frac 1{|\mathbf S^{n-1}|}\int_{E_m}\Omega d\sigma\bigg].
$$
  Set
  $$\Gamma=\{m\in\mathbb N:\sigma(E_m)>2^{-4m}\}\ \text{
and}\ \
\Omega_0=\Omega-\sum_{m\in\Gamma}\|\Omega\|_{L^1(E_m)}\Omega_m.
$$
\begin{lemma}{\rm (see \cite{AP02} or \cite{AACP})}\label{dec of O}
Suppose that $\Omega\in L\log^+\!\!L(\mathbf S^{n-1})$ satisfying \eqref{can of O}. Then $\Omega_0$ and $\Omega_m\ (m\in\Gamma)$ defined above satisfy the following properties:
\item\qquad {\rm (i)} $\int_{\mathbf S^{n-1}}\Omega_md\sigma=0$, $\|\Omega_m\|_{L^1(\mathbf S^{n-1})}\le2$ and $\|\Omega_m\|_{L^2(\mathbf S^{n-1})}\le C2^{2m}$, where $C=C(n)$ is independent of $m$;
\item\qquad {\rm (ii)} $\int_{\mathbf S^{n-1}}\Omega_0d\sigma=0$ and $\Omega_0\in L^2(\mathbf S^{n-1})$;
\item\qquad {\rm (iii)} $\sum_{m\in\Gamma}m\|\Omega\|_{L^1(E_m)}\le \|\Omega\|_{L\log^+\!\!L(\mathbf S^{n-1})}$.
\end{lemma}
We define measure $\nu_j^{(m)}$  by
$$
\nu_j^{(m)}\ast f(x)=\int_{2^j<|y|\le2^{j+1}}\frac{\Omega_m(y)}{|y|^n}f(x-y)dy
$$
for $j\in\mathbb Z$ and $m\in\{0\}\bigcup\Gamma$. Then we have the following pointwise estimate
\begin{align*}
\lambda[N_\lambda^d(\mathscr T^3f)(x)]^{1/2}&\le\sum_{s\ge0}\big(\sum_k\big|[(\delta_0-\phi_k)\ast\nu_{k+s}^{(0)}]\ast
f(x)\big|^2\big)^{1/2}\\
&\;\;\;\;+\sum_{s\ge0}\sum_{m\in\Gamma}\|\Omega\|_{L^1(E_m)}\big(\sum_k\big|[(\delta_0-\phi_k)\ast\nu_{k+s}^{(m)}]\ast
f(x)\big|^2\big)^{1/2}\\
&:=\sum_{s\ge0}G_s^{(0)}f(x)+\sum_{m\in\Gamma}\|\Omega\|_{L^1(E_m)}\sum_{s\ge0}G_s^{(m)}f(x).
\end{align*}
We are reduced to establish estimate of $\|G_s^{(m)}f\|_{L^p(\mathbb R^n)}$ as sharp as possible so that we are able to sum up over $s\geq0$ and $m\in\{0\}\bigcup\Gamma$. Let us start with the case $m\in\Gamma$. We first prove a rapid decay estimate of $\|G_s^{(m)}f\|_{L^2(\mathbb R^n)}$.

For $m\in\Gamma$, Al-Salman and Pan \cite{AP02} proved that
$$
|\widehat{\nu_s^{(m)}}(\xi)|\le
C\min\big\{\|\Omega_m\|_{L^1(\mathbf S^{n-1})},\|\Omega_m\|_{L^2(\mathbf S^{n-1})}[2^s|\xi|]^{-1/3},\|\Omega_m\|_{L^1(\mathbf S^{n-1})}2^s|\xi|\big\}.
$$
By (i) of Lemma \ref{dec of O} and interpolation, we have
\begin{equation}\label{f of vsm}
|\widehat{\nu_s^{(m)}}(\xi)|\le
C\min\{[2^s|\xi|]^{-\frac1{3m}},2^s|\xi|\}.
\end{equation}
It is
trivial that $ |(1-\widehat{\phi})(\xi)|\le\min\{1,|\xi|\}$. To dominate $\|G_s^{(m)}f\|_{L^2(\mathbb R^n)}$, we use above two
estimates and get
\begin{align*}
&\sum_k\big|(1-\widehat{\phi})(2^k\xi)\big|^2\big|\widehat{\nu^{(m)}}(2^{s+k}\xi)\big|^2
\\
&\le\sum_{2^k\le\frac1{2^s|\xi|}}(2^{s+k}|\xi|)^2(2^k|\xi|)^2
+\sum_{\frac1{2^s|\xi|}<2^k\le\frac1{|\xi|}}(2^{s+k}|\xi|)^{-\frac2{3m}}(2^k|\xi|)^2
+\sum_{2^k\ge\frac1{|\xi|}}(2^{s+k}|\xi|)^{-\frac2{3m}}\\
&\le 2^{-2s}+2^{\frac2{3m}}(2^{-\frac{2s}{3m}}-2^{-2s})+2^{-\frac{2s}{3m}}\\
&\le C2^{\frac2{3m}}2^{-\frac{2s}{3m}}.
\end{align*}
Plancherel's theorem implies that
\begin{align}\label{l2 estimate of } \|G_s^{(m)}f\|_{L^2(\mathbb R^n)}\le
C2^{\frac1{3m}}2^{-\frac{s}{3m}}\|f\|_{L^2(\mathbb R^n)}.
\end{align}

To proceed with the estimate of $\|G_s^{(m)}f\|_{L^p(\mathbb R^n)}$, we need three well known or easily checked lemmas below, which will be used in the following proof.
\begin{lemma}{\rm(\cite[p.\,157]{AP02})}\label{bs0} Let $\{\sigma_j\}_{j\in\mathbb{Z}}$ be a
sequence of finite Borel measures. Suppose that there are $\gamma\in(0,1]$ and $p_0\in(2,\infty)$ such that
\item\qquad {\rm (i)} $\|\sigma_j\|\le C$;
\item\qquad {\rm (ii)} $|\widehat{\sigma_j}(\xi)|\le C\min\{|2^j\xi|^\gamma,|2^j\xi|^{-\gamma}\}$;
\item\qquad {\rm (iii)} $\big\|\big(\sum_j|\sigma_j*g_j|^2\big)^{1/2}\big\|_{L^{p_0}}\le C\big\|\big(\sum_j|g_j|^2\big)^{1/2}\big\|_{L^{p_0}}. $

Then, for $p\in(p_0',p_0)$, there exists a constant $C_p$ such that
$$\big\|\big(\sum_j|\sigma_j*f|^2\big)^{1/2}\big\|_{L^{p}}\le C_p2^{\alpha_p\gamma}\big(\frac{2^{\alpha_p\gamma}+1}{2^{\alpha_p\gamma}-1}\big)\|f\|_{L^{p}}, $$
where
$$\alpha_p=\bigg\{ \begin{array}{ll}
 (\frac1p-\frac1{p_0})/(\frac12-\frac1{p_0}),&\qquad p>2,\\
(\frac1p-\frac1{p_0'})/(\frac12-\frac1{p_0'}),&\qquad  p\le2.
\end{array}
$$
\end{lemma}
\begin{lemma}{\rm(\cite[p.\,544]{DR86})}\label{bs} Suppose that $\{\sigma_j\}_{j\in\mathbb{Z}}$ is a
sequence of finite Borel measures. If the maximal operator
$\sigma^{*}(f)=\displaystyle\sup_j||\sigma_j|*f|$ is bounded on
$L^{p_0}$ for $1<p_0\le\infty$, then
$$\big\|\big(\sum_j|\sigma_j*g_j|^2\big)^{1/2}\big\|_{L^p}\le C\big\|\big(\sum_j|g_j|^2\big)^{1/2}\big\|_{L^p}, $$
where $p$ satisfies $\frac1{2p_0}=\big|\frac12-\frac1p\big|$ and $C$
depends on the $L^{p_0}$ norm of  $\sigma^{*}$.
\end{lemma}
\begin{lemma}\label{bs1} Let $a>1$, $t\in(0,1]$ and $\theta\in(0,1]$. Then
$t^\theta a^{2\theta t}\le C_{a,\theta}(a^{\theta t}-1).$
\end{lemma}

Now we estimate $\|G_s^{(m)}f\|_{L^p(\mathbb R^n)}$. For $1<p<\infty$ and $p\neq2$, there exists a $\theta_1\in(0,1)$ and $1<p_1<\infty$ such that $1/p=(1-\theta_1)/{p_1}+\theta_1/2$. For fixed $p_1$, we choose $p_0\in(2,\infty)$ such that
$p_1\in(p_0',p_0)$. Further, choose $p_2\in(1,\infty)$ so that $1/(2p_2)=|1/2-1/p_0|$.
It is trivial that
\begin{align*}
\big|[(\delta_0-\phi_k)\ast\nu_{k+s}^{(m)}]\big|&\le2\|\nu_s^{(m)}\|\le C\|\Omega_m\|_{L^1(\mathbf S^{n-1})}\le C.
\end{align*}
On the other hand, note that $s\ge0$ and $m\ge1$,
\begin{align*}
|[(\delta_0-\phi_k)\ast\nu_{k+s}^{(m)}]^\wedge(\xi)|&\le C\min\{1,|2^k\xi|\}
\min\{|2^{k+s}\xi|,1,|2^{k+s}\xi|^{-\frac1{3m}}\}\\
&\le C\min\{|2^k\xi|,1,|2^k\xi|^{-\frac1{3m}}\}\\
&\le C\min\{|2^k\xi|^{\frac1{3m}},|2^k\xi|^{-\frac1{3m}}\}.
\end{align*}
By the $L^p$ boundedness of maximal function with rough
kernel (see \cite{CZ}) and Lemma \ref{dec of O}, we have
$$
\big\|\sup_{k}\big||[(\delta_0-\phi_k)\ast\nu_{k+s}^{(m)}]|\ast f\big|\big\|_{L^{p_2}(\mathbb R^n)}\le
C_{p_2}\|\Omega_m\|_{L^1(\mathbf S^{n-1})}\|f\|_{L^{p_2}(\mathbb R^n)}\le C_{p_2}\|f\|_{L^{p_2}(\mathbb R^n)},
$$
which, together with Lemma \ref{bs} and Lemma \ref{bs0}, implies
\begin{align}
\nonumber\|G_s^{(m)}f\|_{L^{p_1}(\mathbb
R^n)}&=\|\big(\sum_k\big|[(\delta_0-\phi_k)\ast\nu^{(m)}_{k+s}]\ast
f\big|^2\big)^{1/2}\|_{L^{p_1}(\mathbb R^n)}\\
\label{Gp}& \le C_{p_1}2^{\frac{\alpha_{p_1}}{3m}}\bigg(\frac{2^{\frac{\alpha_{p_1}}{3m}}+1}{2^{\frac{\alpha_{p_1}}{3m}}-1}\bigg)
\|f\|_{L^{p_1}(\mathbb R^n)}\\ \nonumber
&\le C'_{p_1}\bigg(\frac{2^{\frac{2\alpha_{p_1}}{3m}}}{2^{\frac{\alpha_{p_1}}{3m}}-1}\bigg)\|f\|_{L^{p_1}(\mathbb R^n)}\\\nonumber
&\le C''_{p_1}m\|f\|_{L^{p_1}(\mathbb R^n)},
\end{align}
the last inequality was obtained by Lemma \ref{bs1} with $a=2^{\frac{\alpha_{p_1}}{3}}$, $\theta=1$ and $t=1/m$.

Now by interpolation between (\ref{l2 estimate of }) and (\ref{Gp}), we obtain
$$\|G_s^{(m)}f\|_{L^p(\mathbb R^n)}\le C_pm^{1-\theta_1}2^{\frac{\theta_1}{3m}}2^{-\frac{\theta_1 s}{3m}}\|f\|_{L^p(\mathbb R^n)}.$$
Applying Lemma \ref{bs1} with $a=2^{\frac13}$, $\theta=\theta_1$ and $t=1/m$, we conclude that
\begin{align*}
\|\sum_{s\ge0}G_s^{(m)}f\|_{L^p(\mathbb R^n)}&\le C_pm^{1-\theta_1}2^{\frac{\theta_1}{3m}}\sum_{s\ge0}2^{-\frac{\theta_1 s}{3m}}\|f\|_{L^p(\mathbb
R^n)}\\
&\le C_pm^{1-\theta_1}\frac{2^{\frac{2\theta_1}{3m}}}{2^{\frac{\theta_1}{3m}}-1}\|f\|_{L^p(\mathbb
R^n)}\le C_pm\|f\|_{L^p(\mathbb R^n)}.
\end{align*}
In the same way, we can deal with
$G_s^{(0)}f$ for $s\geq0$ and obtain
\begin{align*}
\|\sum_{s\ge0}G_s^{(0)}f\|_{L^p(\mathbb R^n)}\le C_p\sum_{s\ge0}2^{-\frac{\theta_1 s}{3}}\|f\|_{L^p(\mathbb
R^n)}\le C_p\|f\|_{L^p(\mathbb
R^n)}.
\end{align*}

Finally, using Lemma \ref{dec of O}, we get
\begin{align*}
\|\lambda[N_\lambda^d(\mathscr T^3f)]^{1/2}\|_{L^p(\mathbb
R^n)}&\le\|\sum_{s\ge0}G_s^{(0)}f\|_{L^p(\mathbb
R^n)}+\sum_{m\in\Gamma}\|\Omega\|_{L^1(E_m)}\|\sum_{s\ge0}G_s^{(m)}f\|_{L^p(\mathbb
R^n)}\\
&\le C_p\big[1+\sum_{m\in\Gamma}m\|\Omega\|_{L^1(E_m)}\big]\|f\|_{L^p(\mathbb
R^n)}\le C_p[1+\|\Omega\|_{L\log^+L(\mathbf S^{n-1})}]\|f\|_{L^p(\mathbb
R^n)}.
\end{align*}

\medskip

\noindent{\bf Estimate of \eqref{i=1,2,3} for $i=2$.}
Similarly, we have the following pointwise estimate
\begin{align*}
&\lambda[N_\lambda^d(\mathscr T^2f)(x)]^{1/2}\\
&\le\sum_{l<0}\big(\sum_k\big|[\phi_k\ast\nu_{k+l}^{(0)}]\ast
f(x)\big|^2\big)^{1/2}
+\sum_{l<0}\sum_{m\in\Gamma}\|\Omega\|_{L^1(E_m)}\big(\sum_k\big|[\phi_k\ast\nu_{k+l}^{(m)}]\ast
f(x)\big|^2\big)^{1/2}\\
&:=\sum_{l<0}G_l^{(0)}f(x)+\sum_{m\in\Gamma}\|\Omega\|_{L^1(E_m)}\sum_{l<0}G_l^{(m)}f(x).
\end{align*}
For $m\in\Gamma$ and $l<0$, we first estimate the $L^2$ bounds of $G_l^{(m)}f$. Note that $|\hat{\phi}(\xi)|\le 2\chi_{\{|\xi|<4\}}(\xi)$. By Plancherel's theorem and (\ref{f of vsm}),
\begin{align*}
\|G_l^{(m)}f\|_{L^2(\mathbb R^n)}&\le \|f\|_{L^2(\mathbb R^n)}\sup_\xi\big(\sum_k\big|\widehat{\phi}(2^k\xi)\big|^2\big|\widehat{\nu^{(m)}}(2^{l+k}\xi)\big|^2\big)^{1/2}\\
&\le C\|f\|_{L^2(\mathbb R^n)}\sup_\xi\big(\sum_{2^k\le\frac4{|\xi|}}(2^{l+k}|\xi|)^2\big)^{1/2}\le C2^l\|f\|_{L^2(\mathbb R^n)}.
\end{align*}
Next, we consider $\|G_l^{(m)}f\|_{L^p(\mathbb R^n)}$, where $p$ is that in case $i=3$. Obviously, we have
\begin{align*}
\|\phi_k\ast\nu_{k+l}^{(m)}\|&\le C\ \text{and}\
|[\phi_k\ast\nu_{k+l}^{(m)}]^\wedge(\xi)|\le C\min\{|2^k\xi|,|2^k\xi|^{-1}\}.
\end{align*}
In the same way, $
\big\|\sup_{k}\big||\phi_k\ast\nu_{k+l}^{(m)}|\ast f\big|\big\|_{L^{p_2}(\mathbb R^n)}\le C\|f\|_{L^{p_2}(\mathbb R^n)}.
$
The following inequality is a consequence of Lemma \ref{bs0} and Lemma \ref{bs}
\begin{align}\label{Gslm}
\|G_l^{(m)}f\|_{L^{p_1}(\mathbb R^n)}\le C_{p_1}\|f\|_{L^{p_1}(\mathbb
R^n)}.
\end{align}
Interpolating between above $L^{p_1}$ estimate and $L^2$ estimate, and summing over $l$, we get
$$\|\sum_{l<0}G_l^{(m)}f\|_{L^p(\mathbb R^n)}\le C_p\sum_{l<0}2^{\theta_1l}\|f\|_{L^p(\mathbb R^n)}\le C_p\|f\|_{L^p(\mathbb R^n)}.$$
The quantity
$\|\sum_{l<0}G_l^{(0)}f\|_{L^p(\mathbb
R^n)}$ can be treated in the same way. Finally, using Lemma \ref{dec of O}, we get
\begin{align*}
\|\lambda[N_\lambda^d(\mathscr T^2f)]^{1/2}\|_{L^p(\mathbb
R^n)}&\le\|\sum_{l<0}G_l^{(0)}f\|_{L^p(\mathbb
R^n)}+\sum_{m\in\Gamma}\|\Omega\|_{L^1(E_m)}\|\sum_{l<0}G_l^{(m)}f\|_{L^p(\mathbb
R^n)}\\
&\le C_p[1+\|\Omega\|_{L\log^+L(\mathbf S^{n-1})}]\|f\|_{L^p(\mathbb
R^n)}.
\end{align*}
\end{proof}

\section{Proof of Theorem \ref{thm:Omega alpha}}

By Lemma \ref{lem:convert lemma}, to show Theorem \ref{thm:Omega alpha} it suffices to prove the following two propositions.

\begin{proposition}\label{pro:dyadic Omega}
Let $\mathcal T$ be given as in (\ref{tr of S}), $\Omega$ satisfies (\ref{can of O}) and $\Omega\in
\mathcal{G}_\alpha(\mathbf S^{n-1})$ for some $\alpha>0$.Then for $\frac{(2\alpha+1)(1+\alpha)}{2\alpha^2+\alpha+1/2}<p<\frac{2(1+\alpha)(2\alpha+1)}{4\alpha+1}$,
\begin{equation}
\|\lambda\sqrt{N^{d}_{\lambda}(\mathcal T f)}\|_{L^p(\mathbb
R^n)}\leq C_p\|f\|_{L^p(\mathbb R^n)}
\end{equation}
 uniformly in $\lambda>0$.
\end{proposition}

\begin{proposition}\label{pro:short Omega}
Let $\mathcal T$ and $\Omega$ be given as in Proposition \ref{pro:dyadic Omega} but with $\alpha>1$. Then for
$(3+\alpha)/(1+\alpha)<p<(3+\alpha)/2$, we have
\begin{equation}\label{short O p}
\|S_{2}(\mathcal{T}f)\|_{L^p(\mathbb R^n)}\leq C_p\|f\|_{L^p(\mathbb
R^n)}.
\end{equation}
\end{proposition}

\subsection{Proof of Proposition \ref{pro:dyadic Omega}}

We use again Fourier transform and square function estimates. Note that in the present case, the Fourier transform of the measure $\nu$ defined at the beginning of Section 2 has logarithmic decay (see (9) in \cite{GS98}), which is better than the case $\Omega\in L\log^+ L(\mathsf S^{n-1})$, but worse than the case $\Omega\in L^r(\mathsf S^{n-1})$ ($r>1$). Some estimates from \cite{FGP} and \cite{GS99} are taken for granted here. Let us start the proof.

Let measure $\nu_j$ and operator $T_{2^k}$ be defined as in the proof of Lemma \ref{pro:dyadic LlogL} with $\Omega\in\mathcal{G}_\alpha(\mathbf S^{n-1})$. Let $\phi$ be a Schwartz function such that $\hat{\phi}(\xi)=1$ for
$|\xi|\le2$ and $\hat{\phi}(\xi)=0$ for $|\xi|>4$. We have the following decomposition
\begin{align*}
T_{2^k}f&=\phi_k\ast T_{\Omega}f-\phi_k\ast\sum_{l<0}\nu_{k+l}\ast
f+\sum_{s\ge0}(\delta_0-\phi_k)\ast\nu_{k+s}\ast f\\
&:=T^1_kf-T^2_kf+T_k^3f,
\end{align*}
where $\phi_k$ satisfies $\widehat{\phi_k}(\xi)=\hat{\phi}(2^k\xi)$ and $\delta_0$ is the Dirac measure at 0.
$\mathscr T^if$ denotes the family $\{T^i_kf\}_{k\in\mathbb Z}$  for
$i=1,2,3$. Clearly, we need to prove the inequalities below
\begin{align}\label{Ga i=1,2,3}
\|\lambda [N_{\lambda/3}^{d}(\mathscr T^if)]^{1/2}\|_{L^p(\mathbb
R^n)}\le C_p\|f\|_{L^p(\mathbb R^n)},\ \ i=1,2,3,
\end{align}
for $\frac{(2\alpha+1)(1+\alpha)}{2\alpha^2+\alpha+1/2}<p<\frac{2(1+\alpha)(2\alpha+1)}{4\alpha+1}$.

The case $i=1$ of estimate \eqref{Ga i=1,2,3} is just a combination of  Lemma \ref{var of the 1.1} and Theorem 1 in \cite{FGP},
\begin{align*}
\|\lambda[N^d_\lambda(\mathscr T^1f)]^{1/2}\|_{L^p(\mathbb R^n)}\le
C_p\|T_{\Omega}f\|_{L^p(\mathbb R^n)}\leq C_p\|f\|_{L^p(\mathbb R^n)},
\end{align*}
for all $(2+2\alpha)/(1+2\alpha)<p<2+2\alpha$.

Next, we consider the case $i=3$ of estimate \eqref{Ga i=1,2,3}. For $\frac{(2\alpha+1)(1+\alpha)}{2\alpha^2+\alpha+1/2}<p<\frac{2(1+\alpha)(2\alpha+1)}{4\alpha+1}$, there are $\theta_p\in(\frac2{2\alpha+1},1]$ and $p_1\in(\frac{2+2\alpha}{1+2\alpha},2+2\alpha)$ such that
 $$
 \frac1p=\frac{\theta_p}2+\frac{1-\theta_p}{p_1}.
 $$

 The following known result can be found in \cite[p.80]{FGP} and \cite[p.461]{GS98},
\begin{align}\label{2 estimates rough f}
\|G_sf\|_{L^2(\mathbb R^n)}\leq C(1+s)^{-\alpha-1/2}\|f\|_{L^2(\mathbb R^n)},\  \ s>0,
\end{align}
where
$$G_sf(x)=\big(\sum_k\big|[(\delta_0-\phi_k)\ast\nu_{k+s}]\ast
f(x)\big|^2\big)^{1/2}.$$
Note that $|(\delta_0-\phi)^\wedge(\xi)|\le C\min\{1,|\xi|\}$, $|\hat{\nu_s}(\xi)|\le C$, $|\hat{\nu_s}(\xi)|\le C|2^s\xi|$ when $|2^s\xi|\le 2$ and $|\hat{\nu_s}(\xi)|\le C(\ln|2^s\xi|)^{-1-\alpha}$ when $|2^s\xi|\ge 2$. Then we get
\begin{align*}
|[(\delta_0-\phi_k)\ast\nu_{k+s}]^\wedge(\xi)|\le\bigg\{ \begin{array}{ll}
 C|2^k\xi|,&\text{for}\ |2^k\xi|\le2,\\
C(\ln|2^k\xi|)^{-1-\alpha},&\text{for}\ |2^k\xi|\ge2,
\end{array}
\end{align*}
uniformly in $s>0$. Similarly, for $1<p_2<\infty$
$$
\|\sup_{k}||[(\delta_0-\phi_k)\ast\nu_{k+s}]|\ast f|\|_{L^{p_2}(\mathbb R^n)}\le C\|f\|_{L^{p_2}(\mathbb R^n)}.$$ By using the same argument in \cite[p.461]{GS98} or \cite[p.77]{FGP}, we have uniformly in $s>0$  $$ \|G_sf\|_{L^{p_1}(\mathbb
R^n)}\leq C_{p_1}\|f\|_{L^{p_1}(\mathbb R^n)},$$
which, interpolated with \eqref{2 estimates rough f}, implies
$$
\|G_sf\|_{L^{p}(\mathbb
R^n)}\leq C_{p}(1+s)^{-\theta_p(\alpha+1/2)}\|f\|_{L^{p}(\mathbb R^n)}.
$$
Finally, we sum the above $L^p$ estimates over $s>0$ and get
\begin{align*}
\|\lambda[N^{d}_{\lambda}(\mathscr T^3 f)]^{1/2}\|_{L^p(\mathbb
R^n)}&\leq \sum_{s>0}\|G_sf\|_{L^p(\mathbb R^n)}\leq C_{p}\sum_{s>0}(1+s)^{-\theta_p(\alpha+1/2)}\|f\|_{L^{p}(\mathbb R^n)}\leq C_p\|f\|_{L^p(\mathbb R^n)}.
\end{align*}

The case $i=2$  of estimate \eqref{Ga i=1,2,3}  can be treated similarly. For $l\le0$, we set
$$G_lf(x)=\big(\sum_k\big|[\phi_k\ast\nu_{k+l}]\ast
f(x)\big|^2\big)^{1/2}.$$ Note that $\hat{\phi}$
vanishes for $|\xi|\le 2$ and $l\le0$. Using above estimates of $\hat{\nu_s}$, we have
\begin{align*}
\sum_k\big|\widehat{\phi}(2^k\xi)\big|^2\big|\hat{\nu}(2^{l+k}\xi)\big|^2
&\le\sum_{\frac1{|\xi|}<2^k\le\frac1{2^l|\xi|}}(2^{l+k}|\xi|)^2(2^k|\xi|)^{-4}
+\sum_{2^k\le\frac1{|\xi|}}2(2^{l+k}|\xi|)^2\le C2^{2l}.
\end{align*}
Then, $\|G_lf\|_{L^2(\mathbb R^n)}\le C2^l\|f\|_{L^2(\mathbb R^n)}$.
$\|G_lf\|_{L^{p_1}(\mathbb R^n)}\le C\|f\|_{L^{p_1}(\mathbb R^n)}$ can be
proved as (\ref{Gslm}). Interpolating and summing over $l\le0$,
\begin{align*}
\|\lambda[N^d_\lambda(\mathscr T^2f)]^{1/2}\|_{L^p(\mathbb R^n)}\le \sum_{l\le0}\|G_lf\|_{L^p(\mathbb R^n)}\le \sum_{l\le0}2^{\theta_pl}\|f\|_{L^p(\mathbb R^n)}\le
C_p\|f\|_{L^p(\mathbb R^n)}
\end{align*}
for $1<p<\infty$.

\subsection{Proof of Proposition \ref{pro:short Omega}}
In the present case, the rotation method seems not to work. Instead we appeal to the vector-valued singular integral operator theory. That the underlying kernel is rough brings a lot of  trouble, but homogeneity of the kernel is taken advantage of in getting $L^2$-estimate and H\"ormander condition. Let us start the proof.

For $t\in[1,2]$, we define $\nu_{0,t}$ as
$$\nu_{0,t}(x)=\frac{\Omega(x')}{|x|^n}\chi_{\{t\leq|x|\leq2\}}(x)$$
and $\nu_{j,t}(x)={2^{-jn}}\nu_{0,t}(2^{-j}x)$ for $j\in\mathbb{Z}$.
For $k\in\mathbb{Z}$, we define $\Phi_{j,k}f$ by
$\widehat{\Phi_{j,k}f}(\xi)=\varphi(2^{j-k}\xi)\hat{f}(\xi)$, where $\varphi$ is a Schwartz function such that $supp\ (\hat{\varphi})\subset\{1/2\le|\xi|\le 2\}$ and
$\sum_{k\in\mathbb{Z}}\varphi(2^{-k}\xi)=1$ for all
$\xi\in\mathbb{R}^n\setminus\{0\}$. Observe that
$V_{2,j}(\mathcal{T}f)(x)$ is just the strong  $2$-variation function of the family
$\{\nu_{j,t}\ast f(x)\}_{t\in[1,2]}$, hence
\begin{align*}
S_{2}(\mathcal{T}f)(x)&=\Big(\sum_{j\in\mathbb{Z}}|V_{2,j}(\mathcal{T}f)(x)|^2\Big)^{\frac{1}{2}}
=\Big(\sum_{j\in\mathbb{Z}}\|\{\nu_{j,t}\ast f(x)\}_{t\in[1,2]}\|_{V_2}^2\Big)^{\frac{1}{2}}\\
&\leq\sum_{k\in\mathbb{Z}}\Big(\sum_{j\in\mathbb{Z}}\|\{\nu_{j,t}\ast
\Phi_{j,k}f(x)\}_{t\in[1,2]}\|_{V_2}^2\Big)^{\frac{1}{2}}:=
\sum_{k\in\mathbb{Z}}S_{2,k}(\mathcal{T}f)(x).
\end{align*}

The desired estimate \eqref{short O p} will follow from the following two estimates by
interpolation and summation over $k$,
\begin{align}\label{2 estimates variation}
\|S_{2,k}(\mathcal{T}f)\|_{L^2(\mathbb R^n)}\leq
C(1+|k|)^{-\frac{1+\alpha}{2}}\|f\|_{L^2(\mathbb R^n)}
\end{align}
and
\begin{align}\label{p estimates variation}
\|S_{2,k}(\mathcal{T}f)\|_{L^{p}(\mathbb R^n)}\leq
C_{p}(1+|k|)\|f\|_{L^{p}(\mathbb R^n)},
\end{align}
for all $1<p<\infty$

To deal with (\ref{2 estimates variation}), we borrow the fact $\|\mathfrak{a}\|_{V_2}\le\|\mathfrak{a}\|_{L^2}^{1/2}\|\mathfrak{a}'\|_{L^2}^{1/2}$, where $\mathfrak{a}'=\{\frac{d}{dt}a_t:t\in\mathbb R\}$.
It is a special case of (39) in \cite{JSW08}. Then,
\begin{align*}
[S_{2,k}(\mathcal{T}f)(x)]^2\leq\sum_{j\in\mathbb{Z}}\|\nu_{j,t}\ast
\Phi_{j,k}f(x)\|_{L_t^2([1,2])}\|\frac{d}{dt}[\nu_{j,t}\ast
\Phi_{j,k}f(x)]\|_{L^2_t([1,2])}.
\end{align*}
By Cauchy-Schwarz inequality, we have
\begin{equation}\nonumber
\|S_{2,k}(\mathcal{T}f)\|^2_{L^2(\mathbb R^n)}\le
\Big\|\Big(\sum_{j\in\mathbb{Z}}\|\nu_{j,t}\ast
\Phi_{j,k}f\|^2_{L^2_t([1,2])}\Big)^{1/2}\Big\|_{L^2(\mathbb
R^n)}\Big\|\Big(\sum_{j\in\mathbb{Z}}\|\frac{d}{dt}[\nu_{j,t}\ast
\Phi_{j,k}f]\|^2_{L^2_t([1,2])}\Big)^{1/2}\Big\|_{L^2(\mathbb R^n)}.
\end{equation}

To deal with the first term on the right-hand side, we need the following
estimates
\begin{align}\label{f of v a}
|\widehat{\nu_{j,t}}(\xi)|\leq C\left\{\begin{array}{cc}\ln(|2^j\xi|)^{-1-\alpha}& \;\mathrm{if}\; 2^j|\xi|\geq 2\\
       2^j|\xi|& \;\mathrm{if}\; 2^j|\xi|\leq 2\end{array}\right.,
\end{align}
uniformly in $t\in[1,2]$,
which have been essentially proved in \cite{GS98}. By Plancherel's theorem,
\begin{align*}
\Big\|\Big(\sum_{j\in\mathbb{Z}}\|\nu_{j,t}\ast
\Phi_{j,k}f\|^2_{L^2_t([1,2])}\Big)^{1/2}\Big\|_{L^2(\mathbb
R^n)}^2&=\sum_{j\in\mathbb{Z}}\int_1^2\int_{\mathbb
R^n}|\widehat{\nu_{j,t}}(\xi)|^2|\varphi(2^{j-k}\xi)\hat{f}(\xi)|^2d\xi
dt\\
&\le C(1+|k|)^{-2(1+\alpha)}\|f\|^2_{L^2(\mathbb R^n)}.
\end{align*}

In order to treat the second term, we need an elementary fact. That is, for any Schwartz function $h$,
\begin{equation}\label{dft}
|\big(\frac{d}{dt}[\nu_{j,t}\ast h]\big)^\wedge(\xi)|\le C\|\Omega\|_{L^1(\mathbf S^{n-1})}|\hat{h}(\xi)|
\end{equation}
uniformly in $t\in[1,2]$.
Indeed, by spherical coordinate transformation, a trivial calculation shows
\begin{align*}
\frac{d}{dt}[\nu_{j,t}\ast h(x)]&=\frac{d}{dt}\bigg[\int_{2^jt<|y|\le 2^{j+1}}\frac{\Omega(y')}{|y|^n}h(x-y)dy\bigg]=\frac{d}{dt}\bigg[\int_{\mathbf S^{n-1}}\Omega(y')\int_{2^jt}^{2^{j+1}}\frac{1}{r}h(x-ry')drd\sigma(y')\bigg]\\
&=\frac{1}{t}\int_{\mathbf S^{n-1}}\Omega(y')h(x-2^jty')d\sigma(y').
\end{align*}
Note that $t\in[1,2]$,
\begin{align*}
|\big(\frac{d}{dt}[\nu_{j,t}\ast h]\big)^\wedge(\xi)|&\le C\big|\int_{\mathbb R^n}e^{-2\pi ix\cdot\xi}\int_{\mathbf S^{n-1}}\Omega(y')h(x-2^jty')d\sigma(y')dx\big|\\
&\le C\int_{\mathbf S^{n-1}}|\Omega(y')|\big|\int_{\mathbb R^n}e^{-2\pi ix\cdot\xi}h(x-2^jty')dx\big|d\sigma(y')\\
&\le C\|\Omega\|_{L^1(\mathbf S^{n-1})}|\hat{h}(\xi)|.
\end{align*}
By Plancherel's theorem and \eqref{dft}, we have
\begin{align*}
\Big\|\Big(\sum_{j\in\mathbb{Z}}\|\frac{d}{dt}[\nu_{j,t}\ast
\Phi_{j,k}f]\|^2_{L^2_t([1,2])}\Big)^{\frac12}\Big\|_{L^2(\mathbb
R^n)}^2&\le C\sum_{j\in\mathbb{Z}}\int_{\mathbb
R^n}|\varphi(2^{j-k}\xi)\hat{f}(\xi)|^2d\xi \le
C\|f\|^2_{L^2(\mathbb R^n)}.
\end{align*}
Then we get $L^2$ estimate (\ref{2 estimates variation}).

To achieve the estimate (\ref{p estimates variation}), we use the
vector-valued singular integral theory for convolution
operators. For $t\in[1,2]$ and $j\in\mathbb Z$, we define
$K_k(x)=\{\nu_{j,t}\ast\check{\varphi}_{j-k}(x)\}_{j,t}$,
where $\check{\varphi}_{j-k}(y)=2^{n(k-j)}\check{\varphi}(2^{k-j}y)$ and $K_k$ takes value in the Banach space
$$B=\{a(j,t):\;\|a\|_B:=(\sum_j\|a\|^2_{V_2})^{1/2}<\infty\}.$$
Obviously, $S_{2,k}(\mathcal{T}f)=\|K_k\ast f\|_B$. By the vector-valued
singular integral operator theory,  we have the $L^p$ estimates
$$\|S_{2,k}(\mathcal{T}f)\|_{L^p(\mathbb R^n)}\leq C_p(M_k+N_k)\|f\|_{L^p(\mathbb R^n)},\ \ 1<p<\infty,$$
 where $M_k=(1+|k|)^{-\frac{1+\alpha}{2}}$ is the $L^2$-bound in (\ref{2 estimates variation}) and $N_k$ is any upper
bound for
\begin{align*}
\sup_{y\in\mathbb{R}^n}\int_{|x|>2|y|}\|K_k(x-y)-K_k(x)\|_Bdx.
\end{align*}
To get the desired upper bound $C(1+|k|)$ for above integral, we need the
following facts
$$\|a\|_B\leq(\sum_j\|a\|^2_{V_1})^{1/2}\leq
\sum_j\int^2_1|\frac{d}{dt}a(j,t)|dt
$$
and
$$\frac{d}{dt}\big(\nu_{j,t}\ast\check{\varphi}_{j-k}\big)(x)=\int_{\mathbf S^{n-1}}\frac{\Omega(y')}{t}\check{\varphi}_{j-k}(x-2^jty')d\sigma(y').$$
Therefore,
\begin{align*}
&\int_{|x|>2|y|}\|K_k(x-y)-K_k(x)\|_Bdx\\
\leq&\sum_j\int^2_1\int_{|x|>2|y|}\bigg|\frac{d}{dt}\big(\nu_{j,t}\ast\check{\varphi}_{j-k}\big)(x-y)-\frac{d}{dt}
\big(\nu_{j,t}\ast\check{\varphi}_{j-k}\big)(x)\bigg|dxdt\\
\le&\int^2_1\int_{\mathbf
S^{n-1}}|{\Omega(u')}|\sum_j\int_{|x|>2|y|}|\check{\varphi}_{j-k}(x-y-2^jtu')-\check{\varphi}_{j-k}(x-2^jtu')|dxd\sigma(u')
dt.
\end{align*}
To complete the proof, it suffices to show
$$\sum_j\int_{|x|>2|y|}|\check{\varphi}_{j-k}(x-y-2^jtu')-\check{\varphi}_{j-k}(x-2^jtu')|dx\leq C(1+|k|)$$
uniformly in $k$, $t$ and $u'$. Without loss of generality, we assume
$t=1$ and $u'=\mathbf{1}$. For any fixed $y\in\mathbb R^n$ and $k\in\mathbb Z$, we
divide the sum into two parts $I:=\sum_{2^{j+1}<|y|}$ and $II:=
\sum_{2^{j+1}\geq|y|}$. For the first part, we treat it as follows
\begin{align*}
I&=\sum_{2^{j+1}<|y|}\int_{|x|>\frac{2|y|}{2^{j-k}}}|\check{\varphi}(x-\frac{y+2^j\mathbf{1}}{2^{j-k}})-\check{\varphi}(x-2^k\mathbf{1})|dx\\
&\leq2\sum_{2^{j+1}<|y|}\int_{|x|>\frac{|y|}{2^{j-k+1}}}|\check{\varphi}(x)|dx\leq 2\int_{|x|>2^k}|\check{\varphi}(x)|[\sum_{\frac{2^k|y|}{|x|}<2^{j+1}<|y|}1]dx\\
&\leq
2\int_{|x|>2^k}|\check{\varphi}(x)|\log\ {\frac{|x|}{2^{k-1}}}dx\leq
C.
\end{align*}
For the second part, we use the mean value theorem
\begin{align*}
&II=\sum_{2^{j+1}\geq|y|}\int_{|x|>\frac{2|y|}{2^{j-k}}}|\check{\varphi}(x-\frac{y+2^j\mathbf{1}}{2^{j-k}})-\check{\varphi}(x-2^k\mathbf{1})|dx\\
&\leq C\sum_{2^{j+1}\geq|y|}\min\{1,2^{k-j}|y|\}\leq C(1+|k|).
\end{align*}
This completes the proof of (\ref{p estimates variation}).

\section{Proof of Theorem \ref{thm:maximal}}
As in the proof of Theorem \ref{thm:LlogL}, to show Theorem \ref{thm:maximal}(i), it suffices to prove dyadic $\lambda$-jump estimate and short $2$-variation estimate respectively. Also in the course of proving strong $q$-variation estimate ($q>2$) by the rotation method, we can get short $2$-estimate. Let us start with the dyadic $\lambda$-jump estimate.

\begin{proposition}\label{pro:dyadic maximal}
Let $\Omega\in H^1(\mathbf S^{n-1})$ or $L(\log^+\!\!L)^{1/2}(\mathbf S^{n-1})$. Then
\begin{align}\label{NdM}
\|\lambda\sqrt{N^d_{\lambda}(\mathcal Mf)}\|_{L^p(\mathbb R^n)}\le
C_p\|f\|_{L^p(\mathbb R^n)}, \ 1<p<\infty,
\end{align}
uniformly in $\lambda>0$.
\end{proposition}
\begin{proof}  We assume $\Omega$ satisfies the cancelation condition (\ref{can of O}). Otherwise, we write
\begin{align*}
\Omega(x')&=[\Omega(x')-\frac1{\omega_{n-1}}\int_{\mathbf S^{n-1}}\Omega(y')d\sigma(y')]+\frac1{\omega_{n-1}}\int_{\mathbf S^{n-1}}\Omega(y')d\sigma(y')\\
&:=\Omega_0(x')+C(\Omega,n),
\end{align*}
where $\omega_{n-1}$ denotes the area of $\mathbf S^{n-1}$. Thus,
\begin{align*}
M_tf(x)=\frac1{t^n}\int_{|y|<t}\Omega_0(y')f(x-y)dy+C(\Omega,n)\frac1{t^n}\int_{|y|<t}f(x-y)dy.
\end{align*}
\eqref{NdM} has been established in \cite{JKRW98} with $\Omega=1$. Define  $\sigma_{2^k}(y)={2^{-kn}}\chi_{\{|y|<2^k\}}(y)\Omega(y')$, then $M_{2^k}f=\sigma_{2^k}\ast f$.
 The pointwise domination
\begin{equation}\label{Pd}
\lambda^2N^d_{\lambda}(\mathcal Mf)=\lambda^2N_{\lambda}(\{M_{2^k}f\}_k)\leq \sum_{k}|f\ast\sigma_{2^k}|^2
\end{equation}
 reduces the desired estimate to
\begin{equation}\label{Sq es}
\|(\sum_{k}|f\ast\sigma_{2^k}|^2)^{1/2}\|_{L^p(\mathbb R^n)}\leq C_p\|f\|_{L^p(\mathbb R^n)},
\end{equation}
(\ref{Sq es}) can be showed in the similar way as that in
\cite[p.597]{DFP} for $\Omega\in H^1(\mathbf S^{n-1})$ and \cite[p.698]{AACP} for $\Omega\in L(\log^+L)^{1/2}(\mathbf S^{n-1})$. We omit those details for simplicity.
\end{proof}

\begin{remark}
The square function having appeared in the proof of Proposition \ref{pro:dyadic maximal} is actually a discrete analogue of the Marcinkiewicz integral with rough kernels. In \cite{Wal}, Walsh proved that Marcinkiewicz integrals may not be $L^2(\mathbb R^n)$-bounded if the kernel $\Omega\in L(\log^+L)^{1/2-\varepsilon}(\mathbf S^{n-1})$ for $0<\varepsilon<1/2$. Hence it is hopeless to obtain (\ref{NdM}) by using the pointwise domination (\ref{Pd}) with $\Omega\in L^1(\mathbf S^{n-1})$.
\end{remark}

To apply the rotation method, we need the following result in one dimension.

\begin{lemma}\label{pro:1-dim maximal}
Let $\mathscr{M}=\{\mathfrak{M}_t\}_{t>0}$ with $\mathfrak{M}_t$ defined as
\begin{equation}\nonumber
\mathfrak{M}_tf(x)=\frac1{t^n}\int_0^tf(x-s)s^{n-1}ds.
\end{equation}
Then for $1<p<\infty$,
\begin{align*}
\|\lambda\sqrt{N_{\lambda}(\mathscr Mf)}\|_{L^p(\mathbb R)}\le
C_p\|f\|_{L^p(\mathbb R)},
\end{align*}
uniformly in $\lambda>0$. Whence $\|V_q(\mathscr Mf)\|_{L^p(\mathbb R)}\le C_p\|f\|_{L^p(\mathbb R)}$ for $q>2$ and $1<p<\infty$.
\end{lemma}

\begin{proof}
Note that $\mathfrak{M}_tf(x)=\mu_t\ast f(x)$, where $\mu_t(s)=\frac 1t\mu(\frac st)$ and $\mu(s)=\chi_{\{0\le s\le1\}}(s)s^{n-1}$. Van der Corput's lemma implies that
\begin{align*}
|\hat{\mu}(\xi)|\le\bigg|\int_0^1e^{-2\pi ir\xi}r^{n-1}dr\bigg|\le
C|\xi|^{-1}.
\end{align*}
By Theorem 1.1 in \cite{JSW08}, we obtain the following dyadic $\lambda$-jump inequality
\begin{align*}
\|\lambda\sqrt{N^d_{\lambda}(\mathscr Mf)}\|_{L^p(\mathbb R)}\le
C_p\|f\|_{L^p(\mathbb R)},
\end{align*}
uniformly in $\lambda>0$. Further,
by Lemma 6.1 of \cite{JSW08}, we get the $L^p$ boundedness of the short $2$-variation operator
\begin{equation}\label{S2M}
\|S_2(\mathscr Mf)\|_{L^p(\mathbb R)}\le C_p\|f\|_{L^p(\mathbb R)},\
\ 1<p<\infty.
\end{equation}
Above two estimates and Lemma \ref{lem:convert lemma} imply our desired $\lambda$-jump inequality. Clearly, the $q$-variation inequality is a trivial consequence of Lemma \ref{lemma1}.
\end{proof}
Finally, let us turn to the proof of Theorem \ref{thm:maximal}.
\begin{proof}
We first prove Theorem \ref{thm:maximal}(ii) by using the rotation method applied in Section 2.
We define another family of averaging operators $\mathcal M^1=\{M^1_t\}_{t>0}$ with
$$
M^1_tf(x)=\frac1{t^n}\int_0^tf(x-s\mathbf 1)s^{n-1}ds,
$$
where
$\mathbf{1}=(1,0,\cdots,0)$. Recall the rotation of a function is defined as $(R_{\sigma}
f)(x)=f(\sigma x)$. Then, we have
\begin{equation}\nonumber
M_tf(x)=\int_{\mathbf
S^{n-1}}\Omega(y')\frac1{t^n}\int_0^tf(x-sy')s^{n-1}dsd\sigma(y')=\int_{SO(n)}(R_{\sigma^{-1}}M_t^1R_{\sigma} f)(x)\Omega(\sigma\mathbf 1)d\sigma.
\end{equation}
Using Minkowski's inequality and Proposition \ref{pro:1-dim maximal}, we get for $1<p<\infty$
\begin{align*}
\|V_q(\mathcal M f)\|_{L^p(\mathbb R^n)}\le \int_{SO(n)}\|R_{\sigma^{-1}}V_q(\mathcal M^1R_{\sigma} f)\|_{L^p(\mathbb R^n)}|\Omega(\sigma\mathbf
1)|d\sigma\le C_p\|f\|_{L^p(\mathbb R^n)}.
\end{align*}

By Lemma \ref{lem:convert lemma} and Proposition \ref{pro:dyadic maximal}, our desired result Theorem \ref{thm:maximal}(i) follows from $$\|S_2(\mathcal Mf)\|_{L^p(\mathbb R^n)}\le C_p\|f\|_{L^p(\mathbb R^n)},$$ which is consequence of (\ref{S2M}) by rotation method.
\end{proof}

\par

\noindent \textbf{Acknowledgement.} The first author is supported by NSFC (No.11371057, 11471033), SRFDP (No.20130003110003) and the Fundamental Research Funds for the Central Universities (No. 2014KJJCA10). The second author is supported by MINECO: ICMAT Severo Ochoa project SEV-2011-0087 and ERC Grant StG-256997-CZOSQP (EU). The third author is supported by NSFC (No.11371057, 11401175).

\bibliographystyle{amsplain}

\end{document}